\documentclass[12pt, oneside]{article}   	
\usepackage{geometry}               		
\usepackage{graphicx}
	
\usepackage{amssymb}
\usepackage{subcaption}
\usepackage{fancyhdr}
\usepackage{amsmath}
\usepackage{mathtools}
\usepackage{changepage}
\usepackage{dsfont}
\usepackage{float}
\usepackage{bm}
\usepackage{amsthm}
\usepackage{color}
\usepackage{enumerate} 
\usepackage{amsmath}
\usepackage{authblk}
\usepackage{subcaption}
\usepackage{float}

\usepackage{tikz}
\usetikzlibrary{shapes}
\usetikzlibrary{topaths}
\usetikzlibrary{patterns}
\usetikzlibrary{decorations.pathmorphing}

\graphicspath{ {Figs/} }

\newtheorem{defn}{Definition}
\newtheorem{thm}{Theorem}
\newtheorem{lem}{Lemma}

\newtheorem{prop}{Proposition}
\newtheorem{cor}{Corollary}
\newtheorem{claim}{Claim}

\DeclareMathOperator{\diam}{diam}
\DeclareMathOperator{\cdist}{cdist}
\DeclareMathOperator{\supp}{supp}
\newcommand{\cd}{\mathcal{CD}}
\DeclareMathOperator{\conv}{conv}
\DeclareMathOperator{\Ker}{Ker}
\newcommand{\R}{\mathds{R}}

\DeclareMathOperator{\fstabOp}{FSTAB}
\newcommand{\fstab}{\fstabOp}
\DeclareMathOperator{\matchOp}{MATCH}
\newcommand{\match}{\matchOp}
\DeclareMathOperator{\permatchOp}{PERFECT\,MATCH}
\newcommand{\permatch}{\permatchOp}
\DeclareMathOperator{\tsOp}{TS}
\newcommand{\ts}{\tsOp}

\author[1]{Sean Kafer}
\author[2]{Kanstantsin Pashkovich}
\author[1]{Laura Sanit\`a}

\affil[1]{Department of Combinatorics and Optimization, University of Waterloo, 
Canada. \texttt{\{skafer, lsanita\}@uwaterloo.ca}}

\affil[2]{Simons Institute for the Theory of Computing, UC Berkeley, 
USA. \texttt{kpashkov@berkeley.edu}}
\begin{document}

\title{On the Circuit Diameter of some Combinatorial Polytopes}
\maketitle

\abstract{
The \emph{combinatorial diameter} of a polytope $P$ is the maximum value of a shortest path between two vertices of $P$, where the path uses the edges of $P$ only. In contrast to the combinatorial diameter, the \emph{circuit diameter} of $P$ is defined as the maximum value of a shortest path between two vertices of $P$, where the path uses \emph{potential} edge directions of $P$ i.e., all edge directions that can arise by translating some of the facets of $P$. 

In this paper, we study the circuit diameter of polytopes corresponding to classical combinatorial optimization problems, such as the Matching polytope, the Traveling Salesman polytope and the 
Fractional Stable Set polytope.
}

\section{Introduction}
For a polytope $P \subseteq \mathbb R^d$, the \emph{1-skeleton} of $P$ is the graph given by the set of vertices 
($0$-dimensional faces) of $P$, and the set of edges ($1$-dimensional faces) of $P$. The \emph{combinatorial diameter}
of $P$ is the maximum shortest path distance between two vertices in this graph.
Giving bounds on the combinatorial diameter of polytopes is a central question in discrete mathematics
and computational geometry. Combinatorial diameter is fundamental to the theory of linear programming due to the long standing open question about existence of a pivoting rule that yields a polynomial runtime for the Simplex algorithm. Indeed, existence of such a pivoting rule requires a general polynomial bound on the combinatorial diameter of a polytope.

The most famous conjecture in this context is the \emph{Hirsch Conjecture}, proposed in 1957, which
states that the combinatorial diameter of any $d$-dimensional polytope with $f$ facets is at most $f-d$.  
While this conjecture has been disproved for both unbounded polytopes \cite{KW67} and bounded ones \cite{Santos2012}, its \emph{polynomial} version is still open i.e., it is not known whether there is some polynomial function of $f$ and $d$ which upper bounds the combinatorial diameter in general. Currently the best known upper bound on the diameter is exponential in $d$  \cite{S16}.

Recently researchers started investigating whether the bound $f-d$ is a valid upper bound for some different (more powerful) notions of diameter for polytopes. The present work is concerned with one such notion of diameter: the \emph{circuit diameter} of a polytope, formalized by Borgwardt et al. \cite{BFH15}.
Given a polytope of the form $P=\{\bm{x}\in\R^n:A\bm{x}=\bm{b}$,\,  $B\bm{x}\leq \bm{d}\}$ for some rational matrices $A$ and $B$ and rational vectors $\bm{b}$ and $\bm{d}$, the \textit{circuits} of $P$ are the set of \emph{potential}
edge directions that can arise by varying $\bm{b}$ and $\bm{d}$ (see Section \ref{sec:preliminaries} for a formal definition). Starting from a point in $P$ one is allowed to move along any circuit direction until the boundary of $P$ is reached (see Section \ref{sec:preliminaries} for a formal definition). 
Since for every polytope the set of circuit directions contains all edge directions, the combinatorial diameter is always an upper bound on the circuit diameter. Thus even if the Hirsch Conjecture does not hold for the combinatorial diameter, its analogue may be true for the circuit diameter. In particular, Borgwardt et al.\cite{BFH15} conjectured that the circuit diameter is at most $f-d$ for every $d$-dimensional polytope with $f$ facets. We refer the reader to \cite{BSY16} for
recent progress on this conjecture.

Besides studies of upper bounds on combinatorial diameter for general polytopes, there is a long history of studies of such upper bounds for some special classes of polytopes. In particular, many researchers have investigated the combinatorial diameter of polytopes corresponding to classical combinatorial optimization problems. Prominent examples of these polytopes for which the combinatorial diameter have been widely studied are Transportation and Network Flow polytopes \cite{B84,BalinskiRussakoff74,BFH16,BDF16,BHS06}, Matching polytopes \cite{BalinskiRussakoff74,Chvatal75}, 
Traveling Salesman (TSP) polytopes \cite{RC98,GP86}, 
and many others. In this context, there are some questions and conjectures regarding the tightness of the developed bounds which are open, and it is natural to investigate them using a more powerful notion of diameter, like the circuit diameter. The authors of \cite{BFH15} gave upper bounds on the circuit diameter of
Dual Transportation polytopes on bipartite graphs, and later in \cite{BFH16} gave upper bounds on the circuit diameter of Dual Network flow polytopes. 

\paragraph{Our results.} In this paper, we study the circuit diameter of the Matching polytope, the Perfect Matching polytope, 
the TSP polytope, and the Fractional Stable Set polytope. 

Our first result (in Section \ref{sec:matching}) is an exact characterization of the circuit diameter of the Matching polytope (resp., Perfect Matching polytope), which is the convex hull of characteristic vectors of matchings (resp., perfect matchings) in a complete graph with $n$ nodes. In particular, it is well-known that the combinatorial diameter of the Matching polytope equals $\lfloor \frac{n}{2} \rfloor$ \cite{BalinskiRussakoff74,Chvatal75}. In Section \ref{sec:matching}, we show that the circuit diameter of the Matching polytope is upper bounded by a constant in contrast to the combinatorial diameter. In particular, we show that the circuit diameter of the Matching polytope equals $2$ for all $n \geq 7$.  To this aim, we show that for any two different matchings such that one is not contained in the other, the corresponding two vertices are one circuit step away from each other or the corresponding vertices have a common neighbour vertex in the Matching polytope, and therefore their circuit distance is always at most~$2$.  For the Perfect Matching polytope, we show that if $n\neq 8$ the circuit diameter is $1$; and if $n=8$ the circuit diameter is $2$. In contrast, the combinatorial diameter of the Perfect Matching polytope is known to be $2$ for all $n\geq 8$ \cite{PR74}.

In Section \ref{sec:TSP}, we give an exact characterization of the circuit diameter of the TSP polytope, which is the convex hull of all tours (i.e., Hamiltonian cycles) in a complete graph with $n$ nodes. 
It is known that the combinatorial diameter of the TSP polytope is at most 4 \cite{RC98}. In fact, Gr\"otschel and Padberg conjectured in \cite{GP86} that the combinatorial diameter of the TSP polytope is at most $2$, and this conjecture is still open after more than 30 years.
In Section \ref{sec:TSP}, we show that this conjecture holds for the circuit diameter.  In fact, the circuit diameter of the TSP polytope equals $1$ whenever $n\neq 5$; while for $n=5$ the circuit diameter is $2$.
This result is proven by showing that for every two tours in a complete graph, the corresponding vertices are one circuit step from each other whenever $n>5$. Note that no linear description of the TSP polytope is known for general graphs. We achieve the above results for the TSP polytope by using only two famous classes of its facets: namely, \emph{subtour} inequalities and (certain) \emph{comb} inequalities \cite{Grotschel79}.

Finally, we consider the Fractional Stable Set polytope in Section \ref{sec:fSTAB}. This is the polytope given by the standard LP relaxation of the stable set problem for a graph $G$ with $n$ nodes. The Fractional Stable Set polytope was widely studied. In particular, it is known that this polytope is half-integral \cite{B70}, and that the vertices of this polytope have a nice graph interpretation: namely, they can be mapped to subgraphs of $G$ with all connected components being \emph{trees} and \emph{$1$-trees}\footnote{A $1$-tree is a tree plus one edge between two nodes spanned by the tree.}~\cite{CC09, CMN12}. 
This graphical interpretation of vertices was used in \cite{Michini2014} to prove that the combinatorial diameter of the Fractional Stable Set polytope is upper bounded by $n$.  In Section \ref{sec:fSTAB}, we provide a characterization for circuits of this polytope. Specifically, we show that every circuit corresponds to a connected (non necessarily induced) bipartite subgraph of $G$. Our characterization allows us to show that the circuit diameter of the Fractional Stable Set polytope can be essentially upper bounded by the \emph{diameter} of the graph $G$, which is significantly smaller than $n$ in many graphs.

\section{Preliminaries}
\label{sec:preliminaries}
Let $P$ be a polytope of the form $P=\{\bm{x}\in\R^n:A\bm{x}=\bm{b},\, B\bm{x}\leq \bm{d}\}$ for rational matrices $A$ and $B$ and rational vectors $\bm{b}$ and $\bm{d}$.  Let $\Ker(A)$ denote the kernel of $A$ i.e., $\Ker(A):=\{\bm y\in \R^n : A y={\bf 0}\}$. Furthermore, we denote by  $\supp(\bm{x})$ the support of a vector $\bm{x}$.

When talking about the circuit diameter of a polytope $P$, unless specified we assume  that the system of inequalities describing~$P$ is \emph{minimal} with respect to its constraints i.e., each inequality of the above system defines a facet of~$P$. Note that in contrast to the combinatorial diameter, the circuit diameter depends on the linear description of a polytope.  In fact, redundant inequalities might become facet-defining after translating the corresponding hyperplanes. 

\begin{defn}
A non-zero vector $\bm g \in \R^n$ is a circuit of $P$  if 
\begin{enumerate}[(i)]
	\item $\bm g \in \Ker(A)$
	\item $\supp(B \bm g)$ is not contained in any of the sets from the collection $\{\supp(B \bm y): \bm y \in \Ker(A), \bm y\neq {\bf 0} \}$. (i.e., $B\bm g$ is support-minimal in the collection $\{B \bm y : \bm y \in \Ker(A), \bm y \neq 0 \}$)
\end{enumerate}
\end{defn}

Note that if $\bm{c}$ is a circuit of $P$, so is $-\bm{c}$. Given the notion of circuits, we can formally define \emph{circuit steps}, \emph{circuit walks}, and \emph{circuit distance}. 

\begin{defn}
Given $\bm{x}' \in P$, we say that $\bm{x}'' \in P$ is one \textit{circuit step} from $\bm{x}' $, if $\bm{x}''=\bm{x}'+\alpha\bm{c}$ where $\bm{c}$ is a circuit of $P$ and $\alpha>0$ is chosen to be as large as possible so that $\bm{x}'+\alpha\bm{c} \in P$.
\end{defn}

Note that this definition does not specify that $\bm{x}'$ or $\bm{x}''$ are vertices of $P$.

\begin{defn}
Given two points $\bm{x}'$ and $\bm{x}''$ in $P$, a circuit walk from $\bm{x}'$ to $\bm{x}''$ is a sequence of points in $P$, $\bm{x}'=\bm{z}^0, \bm{z}^1,\cdots,\bm{z}^{l-1},\bm{z}^l=\bm{x}''$, where $\bm{z}^i$ is one circuit step from $\bm{z}^{i-1}$, for all $i=1,\cdots,l$.  We say such a circuit walk has length $l$.
\end{defn}

\begin{defn}
Given two points $\bm{x}'$ and $\bm{x}''$ in $P$, the circuit distance from $\bm{x}'$ to $\bm{x}''$, called $\cdist(\bm{x}',\bm{x}'')$, is the length of a shortest circuit walk from $\bm{x}'$ to $\bm{x}''$.
\end{defn}

Note that from the latter two definitions, it follows that a circuit walk from $\bm{x}'$ to $\bm{x}''$ might not always be reversible. For example, let two points $\bm{x}'$ and $\bm{x}''$ be such that $\bm{x}''$ is one circuit step from $\bm{x}'$ i.e., we have that $\bm{x}''=\bm{x}'+\alpha\bm{c}$ and $\alpha>0$ is as large as possible so that $\bm{x}'+\alpha\bm{c}\in P$. However, it may be the case that $\bm{x}''+\alpha'(-\bm{c})\in P$ for some $\alpha'$ such that $\alpha'>\alpha$; and so $\bm{x}'$ is not one circuit step from $\bm{x}''$.  Therefore, it may be the case that $\cdist(\bm{x}',\bm{x}'')\neq \cdist(\bm{x}'',\bm{x}')$. We refer to \cite{F15} for an extensive discussion about circuit distance.

\begin{defn}
Given a polytope $P$, the circuit diameter of $P$, or $\cd(P)$, is the maximum circuit distance between any pair of vertices of $P$.
\end{defn}

Given a system of linear equations $\{A\bm{x} =\bm{0}\,, B\bm{x} = \bm{0}\}$, we say that a vector  $\bm{c}$ is a unique (up to scaling) solution of the system, if every vector $\bm{y}$ satisfying $A\bm{y} =\bm{0}\,, B\bm{y} = \bm{0}$ is of the form $\bm{y} = \lambda\bm{c}$ for some $\lambda \in \R$. 
The following proposition gives an alternative definition of circuits, that will be useful later. It is an easy corollary of the results in \cite{F15}, we report a proof here for completeness.

\begin{prop}\label{prop:system_defn}
Given a polytope $P=\{\bm{x}\in\R^n: A\bm{x} =\bm{b}\,, B\bm{x} \le \bm{d}\}$, a non-zero vector $\bm{c}\in\R^n$ is a circuit if and only if $\bm{c}$ is  a unique (up to scaling) non-zero solution of $\{A \bm{y}=\bm{0}\,, B'\bm{y}=\bm{0}\}$ where $B'$ is a submatrix of $B$.
\end{prop}

\begin{proof}
Let us be given a non-zero vector $\bm c$ such that $A \bm c=0$. Let $B'$ be the maximal (with respect to the number of rows) submatrix of $B$ such that $B'\bm{c}=\bm{0}$. Since $P$ is a polytope the block matrix
$$
\begin{pmatrix}
A\\
B
\end{pmatrix}
$$
has full column rank. Hence, there exists no non-zero vector $\bm{d}$, $A\bm{d}=\bm{0}$, $\supp(B\bm{d})\subset \supp(B\bm{c})$ only if there is  a unique (up to scaling) non-zero solution of $\{A \bm{y}=\bm{0}\,, B'\bm{y}=\bm{0}\}$.

Now, let $B'$ be a submatrix of $B$ such that the system $A\bm{y}=\bm{0}\,,B'\bm{y}=\bm{0}$ has a unique (up to scaling) non-zero solution $\bm{c}$.  Suppose for the sake of contradiction that $\bm{c}$ is not a circuit of $P$.  Then there exists a non-zero vector $\bm{d}$ such that $A\bm{d}=\bm{0}$ and $\supp(B\bm{d})\subset \supp(B\bm{c})$.  In particular, this means that $A\bm{d}=\bm{0}\,,B'\bm{d}=\bm{0}$. Hence $\bm{d}$ is a scaling of $\bm{c}$; and thus $\bm{c}$ is a circuit as desired.
\end{proof}

The next lemma will be used in Section~\ref{sec:matching} to study the circuit diameter of polytopes with linear descriptions, where the coefficients in each inequality are all non-negative or all non-positive.

\begin{lem}\label{lem:nonnegative_circuit}
Let $Q \subseteq\R^n$ be a polytope of the form $Q:=\{\bm{x}\in\R^n:A\bm{x} \leq \bm{b},\, B\bm{x}\leq \bm{d}\}$, where all entries of $A$ are non-negative and all entries of $B$ are non-positive.  Then every circuit $\bm{c}\in\R^n$ of $Q$ with $\bm{c}\geq \bm{0}$ or $\bm{c}\leq \bm{0}$ has exactly one non-zero coordinate. 
\end{lem}

\begin{proof}
Suppose that $\bm{c}$ is a circuit of $Q$ which has at least two non-zero coordinates.  We may assume that $\bm{c}\geq \bm{0}$, as the case where $\bm{c}\leq \bm{0}$ is identical.  Then by Proposition~\ref{prop:system_defn}, $\bm{c}$ is the unique (up to scaling) non-zero solution of $A'\bm{y}=\bm{0}$, $B'\bm{y}=\bm{0}$ where $A'$, $B'$ are some submatrices of $A$, $B$ respectively.  Note that since all entries of $A'$ and $\bm{c}$ are non-negative and $A'\bm{c}=\bm{0}$, we have that for every $i\in \supp(\bm{c})$ the $i$-th column of $A'$ equals $\bm{0}$. Analogously, for every $i\in \supp(\bm{c})$ the $i$-th column of $B'$ equals $\bm{0}$.

Let $i$ be any index such that $\bm{c}_i >0$. Define the vector $\bm{d}$ as
$$
	\bm{d}_j:=\begin{cases}
				1 &\text{  if  } j=i\\
				0 &\text{otherwise}
			\end{cases}\,.
$$
Then $\bm{d}$ is also a solution to $A'\bm{y}=\bm{0}$, $B'\bm{y}=\bm{0}$ and is not a scaling of $\bm{c}$, contradicting that $\bm{c}$ is a circuit. 
\end{proof}
 \bigskip

\section{Matching Polytope}
\label{sec:matching}

The Matching polytope is defined as the convex hull of all characteristic vectors of matchings in a complete graph i.e.,
$$
	P_{\match}(n):=\conv \left \{ \chi(M)\,:\, M \text{ is a matching in } K_n \right \}\,,
$$
where $K_n=(V,E)$ denotes a complete graph with $n$ nodes; and $\chi(M) \in \{0,1\}^E$ denotes the characteristic vector of a matching $M$. 

The linear description of the Matching polytope is well-known and is due to Edmonds~\cite{Edmonds65}. In particular, the following linear system constitutes a minimal linear description of $P_{\match}(n)$
\begin{equation}\label{eq:matching_description}
	\begin{aligned}
		&\bm{x}\left(E[S]\right)\le (|S|-1)/2 &&\text{for all } S \subseteq V,\, |S| \text{ is odd},\, |S|\geq 3\\
		&\bm{x}(\delta(v))\le 1 &&\text{for all } v\in V \\
		& \bm{x} \geq \bm{0}\,,
	\end{aligned}
\end{equation}
where $E[S]$ denotes the set of edges with both endpoints in $S$; $\delta(v)$ denotes the set of edges with one endpoint being $v$; and $\bm{x}(F)$ denotes the sum $\sum_{e \in F}  x_e$ for $F\subseteq E$. 

The combinatorial diameter of the Matching polytope $P_{\match}(n)$ equals $\lfloor n/2 \rfloor$ for  all $n\geq 2$~\cite{BalinskiRussakoff74, Chvatal75}. Our next theorem provides the value of the circuit diameter of the Matching polytope $P_{\match}(n)$ for all possible $n$. In particular,  it shows that the circuit diameter of the Matching polytope is substantially smaller than the combinatorial diameter.

\begin{thm}\label{thm:matching}
	For the Matching polytope we have:
	$$\cd (P_{\match}(n))=\begin{cases}
						1 &n=2,3\\
						2 &n=4,5\\
						3 &n=6\\
						2 &n\ge 7\,.
						\end{cases}$$
\end{thm}

\noindent
The rest of the section is devoted to proving Theorem~\ref{thm:matching}. We first recall  the characterization of adjacency of vertices of the Matching polytope. In this paper, we use symbol $\Delta$ to represent the symmetric difference operator.

\begin{lem}[\cite{BalinskiRussakoff74,Chvatal75}]\label{lem:matching1}
Consider matchings $M_1$, $M_2$ in $K_n$, $n\geq 2$.  $\chi(M_1)$ and $\chi(M_2)$ are adjacent vertices of $P_{\match}(n)$ if and only if $(V, M_1\triangle M_2)$ has a single non-trivial connected component\footnote{Trivial components are components consisting of a single node.}.
\end{lem}

The above lemma has a straightforward corollary.

\begin{cor}\label{cor:matching1}
Consider matchings $M_1$, $M_2$ in $K_n$, $n\geq 2$. If $(V, M_1\triangle M_2)$ has a single non-trivial connected component, then $\bm{c}:=\chi(M_1)-\chi(M_2)$ is a circuit of $P_{\match}(n)$.
\end{cor}

The next lemma shows that the set of circuits of the Matching polytope is much richer than the set of its edge directions. In particular, it shows that for two matchings to define a circuit their symmetric difference does not necessarily have to consist of one non-trivial component only. The circuit directions provided by this lemma will be extensively used to construct short circuit walks in the proof of Theorem~\ref{thm:matching}.  

\begin{lem}\label{lem:matching2}
Consider matchings $M_1$, $M_2$ in $K_n$, such that $M_1\not\subseteq M_2$ and $M_2\not\subseteq M_1$. Then either $(V,M_1 \Delta M_2)$ contains at most two (possibly trivial) connected components, or $\bm{c}:=\chi(M_1)-\chi(M_2)$ is a circuit of $P_{\match}(n)$.
\end{lem}
\begin{proof}
 Suppose that $(V, M_1 \Delta M_2)$ contains at least three connected components.
Let us assume for the sake of contradiction that $\bm{c}=\chi(M_1)-\chi(M_2)$ is not a circuit. Since $\bm{c}$ is not a circuit there exists a non-zero vector $\bm{y}$
such that $\supp(D\bm{y}) \subset \supp(D\bm{c})$, where $D$ denotes the constraint matrix of the minimal linear description~\eqref{eq:matching_description} for the Matching polytope.

Since the inequalities $\bm x_e \geq 0$, $e \in E$ are present in the minimal linear description~\eqref{eq:matching_description} and $\supp(D\bm{y}) \subset \supp(D\bm{c})$, we have that  $\bm y_e =0$ for every edge $e$ such that $\bm c_e=0$. Let $e'=\{v_1,v_2\}$ be an edge so that $\bm y_{e'} \neq 0$. Let $C'$ be the connected component of $(V, M_1 \Delta M_2)$ containing the edge $e'$. Without loss of generality, possibly using rescaling of the vector $\bm y$, we can assume $\bm y_{e'}=1$.  By exchanging the roles of $M_1$ with $M_2$ if necessary, we can assume that $\bm c_{e'}=1$.
Note that $C'$ is either a path or a cycle. Moreover, for all nodes $v$ with degree two in $C'$ we have $\bm{c}(\delta(v))=0$. Since $\supp(D\bm{y}) \subset \supp(D\bm{c})$, we have that $\bm{c}(\delta(v))=0$ implies $\bm{y}(\delta(v))=0$, leading to $\bm y_{e} = \bm c_{e}$ for all $e \in C'$.

Now let $e''=\{u_1,u_2\}$ be an edge such that $\bm c_{e''} = -1$. Note that such an edge $e''$ exists since $M_1\not\subseteq M_2$ and $M_2\not\subseteq M_1$. Let $C''$ be the connected component of $(V,M_1 \Delta M_2)$ containing the edge $e''$.  Let us prove that $\bm y_{e} = \bm c_{e}$ for all $e \in C''$. If $C'$ and $C''$ are the same connected component, then this readily follows from the previous paragraph. If not, let $z$ be a node that belongs to a (possibly trivial) connected component $\tilde C$ of $(V, M_1 \Delta M_2)$ different from $C'$ and $C''$. Let $S:=\{z,u_1,u_2,v_1,v_2\}$ and
note that $\bm{c}(E(S))=0$. Since $\supp(D\bm{y}) \subset \supp(D\bm{c})$, we get $\bm{y}(E(S))=0$, implying $\bm y_{e''} = \bm c_{e''} =-1$. As in the previous paragraph, $C''$ is either a path or a cycle, and for all $v \in V$ with degree two in $C''$ we have $\bm{c}(\delta(v))=0$. Since $\supp(D\bm{y}) \subset \supp(D\bm{c})$, necessarily $\bm{y}(\delta(v))=0$, implying $\bm y_{e} = \bm c_{e}$ for all $e \in C''$.

Now let $e'''=\{w_1,w_2\}$ be an edge not in $C'$ and not in $C''$, but in the connected component $C'''$ of  $(V,M_1 \Delta M_2)$, such that  $\bm c_{e'''} \neq 0$. If $\bm c_{e'''} =1$ (resp. $\bm c_{e'''} =-1$), then we take
the set $S:=\{u_1,u_2,z,w_1,w_2\}$, where $z$ is not in $C''$ and not in $C'''$ (resp. $S:=\{v_1,v_2,z,w_1,w_2\}$, where $z$ is not in $C'$ and not in $C'''$). Since $\bm{c}(E(S))=0$ and $\supp(D\bm{y}) \subset \supp(D\bm{c})$, we get that $\bm{y}(E(S))=0$ holds. On the other side, $\bm{y}(E(S))=0$ implies $\bm y_{e'''} = \bm c_{e'''} =1$ (resp. $\bm y_{e'''} =  \bm c_{e'''} =-1$).
Repeating this argument for all edges in the support of $\bm{c}$ we show that $\bm{y}=\bm{c}$, a contradiction.
\end{proof}

With the above lemma at hand, we are ready to prove Theorem~\ref{thm:matching}.

\begin{proof}\emph{(Proof of Theorem \ref{thm:matching})}
The cases $n=2$ and $n=3$ are trivial. Indeed, $P_{\match}(2)$ and $P_{\match}(3)$ are simplices, and thus every two vertices of $P_{\match}(2)$ and $P_{\match}(3)$ form an edge.

For $n\ge 4$, we consider an empty matching $M_1$ and a matching $M_2$ consisting of two edges to establish
$$
\cd (P_{\match}(n))\ge 2\,.
$$
Indeed, $\cdist(\chi(M_1),\chi(M_2))\ge 2$, because $\bm{c}:=\chi(M_2)-\chi(M_1)$ satisfies $\bm{c}\ge 0$ and has two non-zero entries, and thus $\bm{c}$ is not a circuit  by Lemma~\ref{lem:nonnegative_circuit}. Hence, the vertex $\chi(M_1)$ is not one circuit step away from the vertex $\chi(M_2)$, implying $\cd (P_{\match}(n))\ge 2$.

For $n=6$, the lower bound on the circuit diameter can be improved to the one below
$$
\cd (P_{\match}(6))\ge 3\,.
$$
Consider an empty matching $M_1$ and a perfect matching $M_2$. For a walk from $\chi(M_1)$ to $\chi(M_2)$ the first circuit step at the vertex $\chi(M_1)=\bm{0}$ corresponds to a circuit $\bm{c}$ with $\bm{c}\ge \bm{0}$. Thus, by Lemma~\ref{lem:nonnegative_circuit} the first circuit step corresponds to $\bm{c}$ with exactly one non-zero coordinate.  After the first circuit step we get a vertex $\chi(M')$, where $M'$ is a matching consisting of a single edge $e$. Let us prove that $\bm{c}':=\chi(M_2)-\chi(M')$ is not a circuit and thus $\cdist(\chi(M_1),\chi(M_2))\ge 3$.  If $e\in M_2$, the vector $\bm{c}'$ is not a circuit by Lemma~\ref{lem:nonnegative_circuit}. If $e\not\in M_2$,  let $g$ be the edge in $M_2$ having no common vertex with the edge~$e$. Then the vector $\bm{c}'$ is not a circuit, since the vector $D \chi(g)$ has a smaller support than $D \bm{c}'$, where $D$ is the constraint matrix of the linear description~\eqref{eq:matching_description} for $P_{\match}(6)$. Hence, we showed that any circuit step from $\chi(M_1)$ will always end in a vertex $\chi(M')$, which is at least two circuit steps from $\chi(M_2)$, implying $\cd (P_{\match}(6))\ge 3$.

Now let us prove the corresponding upper bounds for $\cd (P_{\match}(n))$, $n\ge 4$. For $n=4$, $n=5$ and two matchings $M_1$ and $M_2$, $(V,M_1\triangle M_2)$ has at most two non-trivial connected components. This fact together with Corollary~\ref{cor:matching1} implies $\cdist(M_1,M_2) \leq 2$.
For $n=6$ and two matchings $M_1$ and $M_2$, $(V,M_1\triangle M_2)$ has at most three non-trivial connected components. Again, this fact together with Corollary~\ref{cor:matching1} implies $\cdist(M_1,M_2) \leq 3$.

For $n\ge 7$, consider the graph $(V, M_1 \Delta M_2)$ given by the symmetric difference of two matchings $M_1$ and $M_2$. If the symmetric difference contains one $e \in M_1$ and one $e' \in M_2$, then by Lemma~\ref{lem:matching2} 
and  Corollary~\ref{cor:matching1}, $\cdist(M_1,M_2)$  is at most 2. 
Otherwise, the subset $F$ of edges of $M_1 \Delta M_2 $ satisfies 
either $F \subseteq M_1$ or $F \subseteq M_2$. 
If $|F| =2$, the results again follows by Corollary~\ref{cor:matching1}. So assume $|F|\geq 3$. 
First, suppose $F \subseteq M_2$. Let $e$ be any edge connecting two endpoints of two distinct edges in $F$,
and let $\tilde M := M_1 \cup \{e\}$. Clearly, $\cdist(M_1,\tilde M) =1$.
 Now we claim that $\bm{c}:=\chi(M_2)-\chi(\tilde M)$ is a circuit. Indeed, $(V, \tilde M \Delta M_2)$ has at least three connected component:  one path of length 3 and either at least two other edges, or one other edge plus at least one trivial connected component consisting of a single node (since $n\geq 7$).
 In both cases, Lemma~\ref{lem:matching2} implies that $\bm{c}:=\chi(M_2)-\chi(\tilde M)$ is a circuit, leading to the result.
Finally, suppose $F \subseteq M_1$. Similarly to the previous case, we set $\tilde M := M_2 \cup \{e\}$.
Then, by Lemma~\ref{lem:matching2} we get that $\chi(\tilde M)-\chi(M_1)$ is a circuit, and 
by  Corollary~\ref{cor:matching1} we get that  $\chi(M_2) - \chi(\tilde M)$ is a circuit, leading to the result.
\end{proof}

\subsection{Perfect Matching Polytope}

Let us define the Perfect Matching polytope
$$
	P_{\permatch}(n):=\conv \left \{ \chi(M)\,:\, M \text{ is a perfect matching in } K_n\right \}\,,
$$
where $n\ge 4$ and $n$ is even. In~\cite{Edmonds65}, Edmonds showed that the following linear system constitutes a minimal linear description of $P_{\permatch}(n)$
\begin{equation}\label{eq:perfect_matching_description}
	\begin{aligned}
		&\bm{x}\left(\delta(S)\right)\ge 1 &&\text{for all } S \subset V,\, |S| \text{ is odd}\,,|S|\geq 3\\
		&\bm{x}(\delta(v))= 1 &&\text{for all } v\in V\, \\
		&\bm{x} \geq 0\,. &&
	\end{aligned}
\end{equation}

\begin{thm}\label{thm:perfect_matching}
	For the perfect matching polytope we have:
	$$\cd (P_{\permatch}(n))=\begin{cases}
						1 &n=4,6\\
						2 &n= 8\\
						1 &n\ge 10\,.
						\end{cases}$$
\end{thm}
The rest of this section is devoted to prove Theorem~\ref{thm:perfect_matching}. First, let us recall the characterization of adjacency of the vertices of the Perfect Matching polytope.

\begin{lem}[\cite{BalinskiRussakoff74,Chvatal75}]\label{lem:perfect_matching1}
Consider perfect matchings $M_1$, $M_2$ in $K_n$, $n\geq 2$.  $\chi(M_1)$ and $\chi(M_2)$ are adjacent vertices of $P_{\permatch}(n)$ if and only if $(V, M_1\triangle M_2)$ has a single non-trivial connected component.
\end{lem}

The above lemma has a straightforward corollary.

\begin{cor}\label{cor:perfect_matching1}
Consider perfect matchings $M_1$, $M_2$ in $K_n$, $n\geq 2$. If $(V, M_1\triangle M_2)$ has a single non-trivial connected component, then $\bm{c}:=\chi(M_1)-\chi(M_2)$ is a circuit of $P_{\permatch}(n)$.
\end{cor}

The next lemma shows that every two different matchings define a circuit  whenever $n\geq 10$. The circuit directions provided by this lemma will be extensively used to construct short circuit walks in the proof of Theorem~\ref{thm:perfect_matching}. The proof of Lemma~\ref{lem:perfect_matching2} uses ideas similar to the ones in the proof of Lemma~\ref{lem:matching2}. 

\begin{lem}\label{lem:perfect_matching2}
Consider two different perfect matchings $M_1$, $M_2$ in $K_n$, $n\geq 10$. Then $\bm{c}:=\chi(M_1)-\chi(M_2)$ is a circuit of $P_{\permatch}(n)$.
\end{lem}
\begin{proof}
Let us assume for the sake of contradiction that $\bm{c}$ is not a circuit. Then there exists a non-zero vector $\bm{y}$ such that $\supp(D\bm{y}) \subset \supp(D\bm{c})$, where $D$ is the constraint matrix of \eqref{eq:perfect_matching_description}. Since the inequalities $\bm x_e \geq 0$, $e \in E$ are in the minimal linear description~\eqref{eq:perfect_matching_description} and $\supp(D\bm{y}) \subset \supp(D\bm{c})$,
we have $\bm y_e =0$ for every edge $e$ such that $\bm c_e =0$. 

Let $e'=\{v_1,v_2\}$ be such that $y_{e'} \neq 0$. Without loss of generality, possibly rescaling vector $\bm y$ we can assume $\bm y_{e'}=1$. Let $C'$ be the connected component of $(V, M_1 \Delta M_2)$ containing $e'$. By exchanging the roles of $M_1$ with $M_2$, we can assume $\bm c_{e'}=1$.
Moreover, for every node $v$ we have $\bm{c}(\delta(v))=0$. Since $\supp(D\bm{y}) \subset \supp(D\bm{c})$, we have $\bm{y}(\delta(v))=0$ for every node $v$. Since $C'$ is an even cycle, $\bm{y}(\delta(v))=0$, $v\in V$ implies $\bm y_{e} = \bm c_{e}$ for all edges $e \in C'$. In particular, for an edge $f = \{v_2, v_3\}$, $f\in (M_1 \Delta M_2)$, which is different from the edge $e'$, we have $\bm y_f =\bm c_f=-1$.

Now let $C''$ be a connected component of $(V,M_1 \Delta M_2)$, different from $C'$. Note that such $C''$ exists since otherwise $(V,M_1 \Delta M_2)$ contains only one non-trivial connected component, implying that $\bm{c}$ is a circuit by Lemma~\ref{lem:perfect_matching1}. Let  $e''=\{u_1,u_2\}$ be an edge in $C''$ such that $\bm c_{e''}=-1$. Again, since $\bm{y}(\delta(v))=0$ for every node $v$ and since $C''$ is an even cycle, there exists $\gamma$ such that $\bm{y}_{e}=\gamma \bm{c}_e$ for every edge $e$ in $C''$.

Let $z$ be a node that is not adjacent to any of the nodes $u_1,u_2,v_1,v_2$ in the graph $(V,M_1 \Delta M_2)$. Note that such a node exists, because each node in $(V,M_1 \Delta M_2)$ has degree exactly $2$, and we have $n>8$. Also note that such node $z$ is not equal to any of the nodes  $u_1,u_2,v_1,v_2$, since $\{u_1,u_2\}$ and $\{v_1,v_2\}$ are edges in $(V,M_1 \Delta M_2)$.
Let us define $S:=\{z,u_1,u_2,v_1,v_2\}$. It is straightforward to check that $\bm{c}(\delta(S))=0$. Indeed, since $\supp(D\bm{y}) \subset \supp(D\bm{c})$ and the constraint $\bm x(\delta(S))\geq 1$ is present in~\eqref{eq:perfect_matching_description}, we have that $\bm{y}(\delta(S))=0$.
On the other side, $\bm{y}(\delta(S))= -2 - 2\gamma = 0$,  implying $\gamma = 1$ and therefore $\bm y_{e} = \bm c_{e}$ for all $e\in C''$. Repeating this argument for all non-trivial connected components of $(V,M_1 \Delta M_2)$, we get $\bm{y}=\bm{c}$, a contradiction.

\end{proof}

Now, with Lemma~\ref{lem:perfect_matching2} at hand, we are ready to prove Theorem~\ref{thm:perfect_matching}.

\begin{proof}\emph{(Proof of Theorem \ref{thm:perfect_matching})}
To show that the corresponding lower bounds for the circuit diameter hold, it is enough to show that
$$P_{\permatch}(8)\ge 2\,.$$
To show this, let us define two perfect matchings in the complete graph $K_8$ with the node set $\{v_1,\ldots, v_8\}$
$$M_1:=\{v_1v_2, v_3v_4, v_5v_6, v_7v_8\}\quad and \quad M_2:=\{v_1v_4, v_3v_2, v_5v_8, v_7v_6\}\,.$$
The vector $\bm{c}:=\chi(M_1)-\chi(M_2)$ is not a circuit, since the vector $D\bm{c}$ has a larger support than $D\left(\chi(\{v_1v_2, v_3v_4\})-\chi(\{v_1v_4, v_3v_2\})\right)$, where $D$ is the linear constraint matrix of the linear description of  $P_{\permatch}(8)$. Hence, we have $$\cd (P_{\permatch}(8))\ge 2\,.$$

Now let us prove the corresponding upper bounds for $\cd (P_{\match}(n))$, $n\ge 4$. For $n=4$, $n=6$ and two perfect matchings $M_1$ and $M_2$, $(V, M_1\triangle M_2)$ has at most one non-trivial connected component. This fact together with Corollary~\ref{cor:perfect_matching1} implies $\cd (P_{\match}(n))\le 1$  for $n=4$, $n=6$.

For $n=8$ and two perfect matchings $M_1$ and $M_2$, $(V,M_1\triangle M_2)$ has at most two non-trivial connected components.  Again, this fact together with Corollary~\ref{cor:perfect_matching1} implies $\cd (P_{\permatch}(8))\le 2$. For $n\ge 10$, the upper bound follows from Lemma~\ref{lem:perfect_matching2}.
\end{proof}

\section{Traveling Salesman Polytope}\label{sec:TSP}

The Traveling Salesman polytope is defined as the convex hull of characteristic vectors of Hamiltonian cycles in a complete graph i.e.,
$$
	P_{\ts}(n):=\conv \left \{ \chi(T)\,:\, T \text{ is a Hamiltonian cycle in } K_n\right \}\,.
$$
In fact, no linear description of the Traveling Salesman polytope is known for general $n$. Moreover any linear description of $P_{\ts}(n)$, which admits an efficient way to test whether a given linear constraint belongs to this description, would have consequences for the long-standing conjecture $\mathcal{NP}=co-\mathcal{NP}$~\cite{Schrijver2003}(Section 5.12).
However, for some small values of $n$ a linear description of the Traveling Salesman polytope is known. For example, $P_{\ts}(5)$ can be described by nonnegativity constraints and the constraints below~\cite{FN92}
\begin{equation}\label{eq:tsp_description_5}
	\begin{aligned}
		&\bm{x}\left(E(S)\right)\le |S|-1 &&\text{for all } S,\,  S \subseteq V,\, 2\leq|S|\leq |V|-2\\
		&\bm{x}(\delta(v))= 2 &&\text{for all } v\in V\\
		&\bm{x}\geq \bm{0}\,.
	\end{aligned}
\end{equation}
Moreover,  the linear inequalities from~\eqref{eq:tsp_description_5} define facets of the Traveling Salesman polytope $P_{\ts}(n)$ for all $n \geq 4$~\cite{Grotschel79}. For $n\ge 6$ the inequalities
\begin{equation}\label{eq:comb}
	\bm{x}_{uv}+\bm{x}_{vw}+\bm{x}_{wu}+\bm{x}_{uu'}+ \bm{x}_{vv'}+\bm{x}_{ww'}\le 4\quad\text{for distinct } u,v,w, u', v', w' \in V
\end{equation}
also define facets of $P_{\ts}(n)$~\cite{Grotschel79}. The inequality~\eqref{eq:comb} belongs to the well-known family of \emph{comb inequalities}, which are valid for the Traveling Salesman polytope. Surprisingly, such scarce knowledge on linear description of the Traveling Salesman polytope is enough for us to prove the following theorem.

\begin{thm}\label{thm:tsp}
For the Traveling Salesman polytope we have:
	$$\cd (P_{\ts}(n))=\begin{cases}
						1 &n=3,4\\
						2 &n=5\\
						1 &n\ge 6\,.
						\end{cases}$$
\end{thm}

The proof of Theorem~\ref{thm:tsp} follows from a series of lemmata below.

\begin{lem}\label{lem:tsp1}
For $n=5$ we have $\cd (P_{\ts}(n))=2$.
\end{lem}
\begin{proof}
Recall, that the Traveling Salesman polytope $P_{\ts}(5)$ admits the minimal linear description~\eqref{eq:tsp_description_5}~\cite{FN92}.

For two Hamiltonian cycles $T_1$, $T_2$ in $K_5$ without a common edge (see Figure~\ref{fig:tsp5 disjoint}), the vector $\bm{c}:=\chi(T_1)-\chi(T_2)$ is not a circuit of $P_{\ts}(5)$. Indeed, $\supp (D\bm y)\subset \supp (D\bm c)$ for the non-zero vector $\bm y:=\chi(M_1)-\chi(M_2)$, where $D$ is the constraint matrix of~\eqref{eq:tsp_description_5} and $M_1$, $M_2$ are two different matchings in $K_5$ on the same four nodes. Thus $\cd (P_{\ts}(5))\ge 2$. 
\begin{figure}[!ht]
	\begin{center}
		\begin{tikzpicture}[xscale=1.2,yscale=1.5]
	
			\tikzset{
				graph node/.style={shape=circle,draw=black,inner sep=0pt, minimum size=4pt
      						  }
					}
			\tikzset{	plus edge/.style={red, thick
      						  }
					}		
			\tikzset{			minus edge/.style={blue, line width=.8mm, dashed
      						  }
					}

			\node[graph node](v1) at (2,1){};
			\node[graph node](v2) at (2,-.5){};
			\node[graph node](v4) at (-2,-.5){};
			\node[graph node](v3) at (0,-1.5){};
			\node[graph node](v5) at (-2,1){};

			\draw[plus edge] (v5)--(v1)--(v2)--(v3)--(v4)--(v5);
			\draw[minus edge] (v5)--(v2)--(v4)--(v1)--(v3)--(v5);
		\end{tikzpicture}
	\end{center}
	\caption{Hamiltonian cycles $T_1$ and $T_2$ in $K_5$ without a common edge. Here, the edges of $T_2$ are depicted as dashed edges.}
	\label{fig:tsp5 disjoint}
\end{figure}
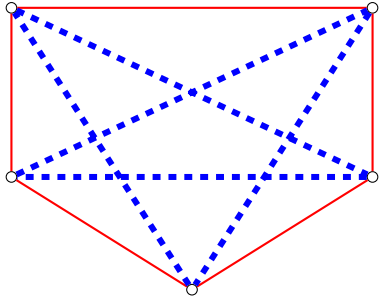

The bound $\cd (P_{\ts}(5))\le 2$ follows from the fact that for any two Hamiltonian cycles $T_1$, $T_2$ such that $T_1\cap T_2\neq \varnothing$, $\chi(T_1)-\chi(T_2)$ is a circuit of $P_{\ts}(5)$. Indeed, up to symmetry we have two possible cases (see Figure~\ref{fig:tsp5 not disjoint}) and in each of these cases $\chi(T_1)-\chi(T_2)$ is a circuit.
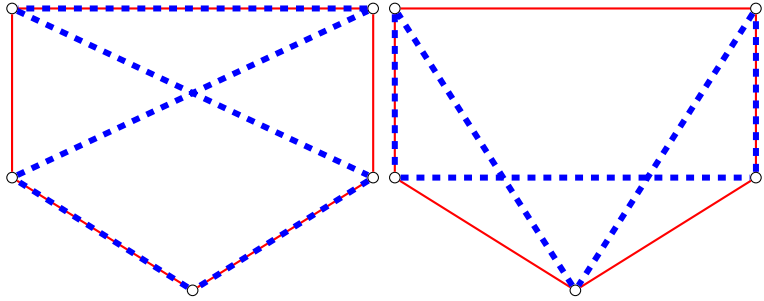
\begin{figure}[!ht]
	\begin{center}
		\begin{tikzpicture}[xscale=1.2,yscale=1.5]
	
			\tikzset{
				graph node/.style={shape=circle,draw=black,inner sep=0pt, minimum size=4pt
      						  }
					}
			\tikzset{	plus edge/.style={red, thick
      						  }
					}		
			\tikzset{			minus edge/.style={blue, line width=.8mm, dashed
      						  }
					}

			\node[graph node](v1) at (2,1){};
			\node[graph node](v2) at (2,-.5){};
			\node[graph node](v4) at (-2,-.5){};
			\node[graph node](v3) at (0,-1.5){};
			\node[graph node](v5) at (-2,1){};

			\draw[plus edge] (v5)--(v1)--(v2)--(v3)--(v4)--(v5);
			\draw[minus edge] (v5)--(v2)--(v3)--(v4)--(v1)--(v5);
		\end{tikzpicture}
				\begin{tikzpicture}[xscale=1.2,yscale=1.5]
	
			\tikzset{
				graph node/.style={shape=circle,draw=black,inner sep=0pt, minimum size=4pt
      						  }
					}
			\tikzset{	plus edge/.style={red, thick
      						  }
					}		
			\tikzset{			minus edge/.style={blue, line width=.8mm, dashed
      						  }
					}

			\node[graph node](v1) at (2,1){};
			\node[graph node](v2) at (2,-.5){};
			\node[graph node](v4) at (-2,-.5){};
			\node[graph node](v3) at (0,-1.5){};
			\node[graph node](v5) at (-2,1){};

			\draw[plus edge] (v5)--(v1)--(v2)--(v3)--(v4)--(v5);
			\draw[minus edge] (v5)--(v4)--(v2)--(v1)--(v3)--(v5);
		\end{tikzpicture}
	\end{center}
	\caption{Hamiltonian cycles $T_1$ and $T_2$ in $K_5$ with a common edge. Here, the edges of $T_2$ are depicted as dashed edges.}
	\label{fig:tsp5 not disjoint}
\end{figure}
\end{proof}

\begin{lem}\label{lem:tsp2}
For $n=6$ we have $\cd (P_{\ts}(n))=1$.
\end{lem}
\begin{proof}
Let us consider two different Hamiltonian cycles $T_1$ and $T_2$ in $K_6$, then up to symmetry and up to exchanging the roles of $T_1$ and $T_2$ we have one of the nine cases (see Figure~\ref{fig:tsp6}). In all these nine cases, $\bm y :=\chi(T_1)-\chi(T_2)$ is a circuit of $P_{\ts}(6)$.
 
 \begin{figure}[!ht]
	\begin{center}
		\begin{tikzpicture}[xscale=.9,yscale=1.2]
	
			\tikzset{
				graph node/.style={shape=circle,draw=black,inner sep=0pt, minimum size=4pt
      						  }
					}
			\tikzset{	plus edge/.style={red, thick
      						  }
					}		
			\tikzset{			minus edge/.style={blue, line width=.8mm, dashed
      						  }
					}

			\node[graph node](v1) at (2,1){};
			\node[graph node](v2) at (2,-.5){};
			\node[graph node](v4) at (-2,-.5){};
			\node[graph node](v3) at (0,-1.5){};
			\node[graph node](v5) at (-2,1){};
			\node[graph node](v6) at (0,2){};

			\draw[plus edge] (v6)--(v1)--(v2)--(v3)--(v4)--(v5)--(v6);
			\draw[minus edge] (v6)--(v3)--(v5)--(v1)--(v2)--(v4)--(v6);
		\end{tikzpicture}
		\begin{tikzpicture}[xscale=.9,yscale=1.2]
	
			\tikzset{
				graph node/.style={shape=circle,draw=black,inner sep=0pt, minimum size=4pt
      						  }
					}
			\tikzset{	plus edge/.style={red, thick
      						  }
					}		
			\tikzset{			minus edge/.style={blue, line width=.8mm, dashed
      						  }
					}

			\node[graph node](v1) at (2,1){};
			\node[graph node](v2) at (2,-.5){};
			\node[graph node](v4) at (-2,-.5){};
			\node[graph node](v3) at (0,-1.5){};
			\node[graph node](v5) at (-2,1){};
			\node[graph node](v6) at (0,2){};

			\draw[plus edge] (v6)--(v1)--(v2)--(v3)--(v4)--(v5)--(v6);
			\draw[minus edge] (v1)--(v2)--(v4)--(v6)--(v5)--(v3)--(v1);
		\end{tikzpicture}
				\begin{tikzpicture}[xscale=.9,yscale=1.2]
	
			\tikzset{
				graph node/.style={shape=circle,draw=black,inner sep=0pt, minimum size=4pt
      						  }
					}
			\tikzset{	plus edge/.style={red, thick
      						  }
					}		
			\tikzset{			minus edge/.style={blue, line width=.8mm, dashed
      						  }
					}

			\node[graph node](v1) at (2,1){};
			\node[graph node](v2) at (2,-.5){};
			\node[graph node](v4) at (-2,-.5){};
			\node[graph node](v3) at (0,-1.5){};
			\node[graph node](v5) at (-2,1){};
			\node[graph node](v6) at (0,2){};

			\draw[plus edge] (v6)--(v1)--(v2)--(v3)--(v4)--(v5)--(v6);
			\draw[minus edge] (v1)--(v3)--(v2)--(v4)--(v6)--(v5)--(v1);
		\end{tikzpicture}

		\begin{tikzpicture}[xscale=.9,yscale=1.2]
	
			\tikzset{
				graph node/.style={shape=circle,draw=black,inner sep=0pt, minimum size=4pt
      						  }
					}
			\tikzset{	plus edge/.style={red, thick
      						  }
					}		
			\tikzset{			minus edge/.style={blue, line width=.8mm, dashed
      						  }
					}

			\node[graph node](v1) at (2,1){};
			\node[graph node](v2) at (2,-.5){};
			\node[graph node](v4) at (-2,-.5){};
			\node[graph node](v3) at (0,-1.5){};
			\node[graph node](v5) at (-2,1){};
			\node[graph node](v6) at (0,2){};

			\draw[plus edge] (v6)--(v1)--(v2)--(v3)--(v4)--(v5)--(v6);
			\draw[minus edge] (v1)--(v5)--(v4)--(v2)--(v3)--(v6)--(v1);
		\end{tikzpicture}		
		\begin{tikzpicture}[xscale=.9,yscale=1.2]
	
			\tikzset{
				graph node/.style={shape=circle,draw=black,inner sep=0pt, minimum size=4pt
      						  }
					}
			\tikzset{	plus edge/.style={red, thick
      						  }
					}		
			\tikzset{			minus edge/.style={blue, line width=.8mm, dashed
      						  }
					}

			\node[graph node](v1) at (2,1){};
			\node[graph node](v2) at (2,-.5){};
			\node[graph node](v4) at (-2,-.5){};
			\node[graph node](v3) at (0,-1.5){};
			\node[graph node](v5) at (-2,1){};
			\node[graph node](v6) at (0,2){};

			\draw[plus edge] (v6)--(v1)--(v2)--(v3)--(v4)--(v5)--(v6);
			\draw[minus edge] (v1)--(v6)--(v3)--(v2)--(v5)--(v4)--(v1);
		\end{tikzpicture}
		\begin{tikzpicture}[xscale=.9,yscale=1.2]
	
			\tikzset{
				graph node/.style={shape=circle,draw=black,inner sep=0pt, minimum size=4pt
      						  }
					}
			\tikzset{	plus edge/.style={red, thick
      						  }
					}		
			\tikzset{			minus edge/.style={blue, line width=.8mm, dashed
      						  }
					}

			\node[graph node](v1) at (2,1){};
			\node[graph node](v2) at (2,-.5){};
			\node[graph node](v4) at (-2,-.5){};
			\node[graph node](v3) at (0,-1.5){};
			\node[graph node](v5) at (-2,1){};
			\node[graph node](v6) at (0,2){};

			\draw[plus edge] (v6)--(v1)--(v2)--(v3)--(v4)--(v5)--(v6);
			\draw[minus edge] (v1)--(v6)--(v3)--(v5)--(v4)--(v2)--(v1);
		\end{tikzpicture}
		\begin{tikzpicture}[xscale=.9,yscale=1.2]
	
			\tikzset{
				graph node/.style={shape=circle,draw=black,inner sep=0pt, minimum size=4pt
      						  }
					}
			\tikzset{	plus edge/.style={red, thick
      						  }
					}		
			\tikzset{			minus edge/.style={blue, line width=.8mm, dashed
      						  }
					}

			\node[graph node](v1) at (2,1){};
			\node[graph node](v2) at (2,-.5){};
			\node[graph node](v4) at (-2,-.5){};
			\node[graph node](v3) at (0,-1.5){};
			\node[graph node](v5) at (-2,1){};
			\node[graph node](v6) at (0,2){};

			\draw[plus edge] (v6)--(v1)--(v2)--(v3)--(v4)--(v5)--(v6);
			\draw[minus edge] (v1)--(v6)--(v5)--(v2)--(v3)--(v4)--(v1);
		\end{tikzpicture}
		
		\vspace{2mm}
		
		\begin{tikzpicture}[xscale=.9,yscale=1.2]
	
			\tikzset{
				graph node/.style={shape=circle,draw=black,inner sep=0pt, minimum size=4pt
      						  }
					}
			\tikzset{	plus edge/.style={red, thick
      						  }
					}		
			\tikzset{			minus edge/.style={blue, line width=.8mm, dashed
      						  }
					}

			\node[graph node](v1) at (2,1){};
			\node[graph node](v2) at (2,-.5){};
			\node[graph node](v4) at (-2,-.5){};
			\node[graph node](v3) at (0,-1.5){};
			\node[graph node](v5) at (-2,1){};
			\node[graph node](v6) at (0,2){};

			\draw[plus edge] (v6)--(v1)--(v2)--(v3)--(v4)--(v5)--(v6);
			\draw[minus edge] (v1)--(v6)--(v5)--(v4)--(v2)--(v3)--(v1);
		\end{tikzpicture}
				\begin{tikzpicture}[xscale=.9,yscale=1.2]
	
			\tikzset{
				graph node/.style={shape=circle,draw=black,inner sep=0pt, minimum size=4pt
      						  }
					}
			\tikzset{	plus edge/.style={red, thick
      						  }
					}		
			\tikzset{			minus edge/.style={blue, line width=.8mm, dashed
      						  }
					}

			\node[graph node](v1) at (2,1){};
			\node[graph node](v2) at (2,-.5){};
			\node[graph node](v4) at (-2,-.5){};
			\node[graph node](v3) at (0,-1.5){};
			\node[graph node](v5) at (-2,1){};
			\node[graph node](v6) at (0,2){};

			\draw[plus edge] (v6)--(v1)--(v2)--(v3)--(v4)--(v5)--(v6);
			\draw[minus edge] (v1)--(v5)--(v2)--(v4)--(v6)--(v3)--(v1);
		\end{tikzpicture}
	\end{center}
	\caption[All possible cases for two different Hamiltonian cycles in $K_6$]{All possible cases (up to symmetry and up to exchanging the roles of $T_1$ and $T_2$) for two different Hamiltonian cycles $T_1$ and $T_2$ in $K_6$. Here, the edges of $T_2$ are depicted as dashed edges.}
	\label{fig:tsp6}
\end{figure}
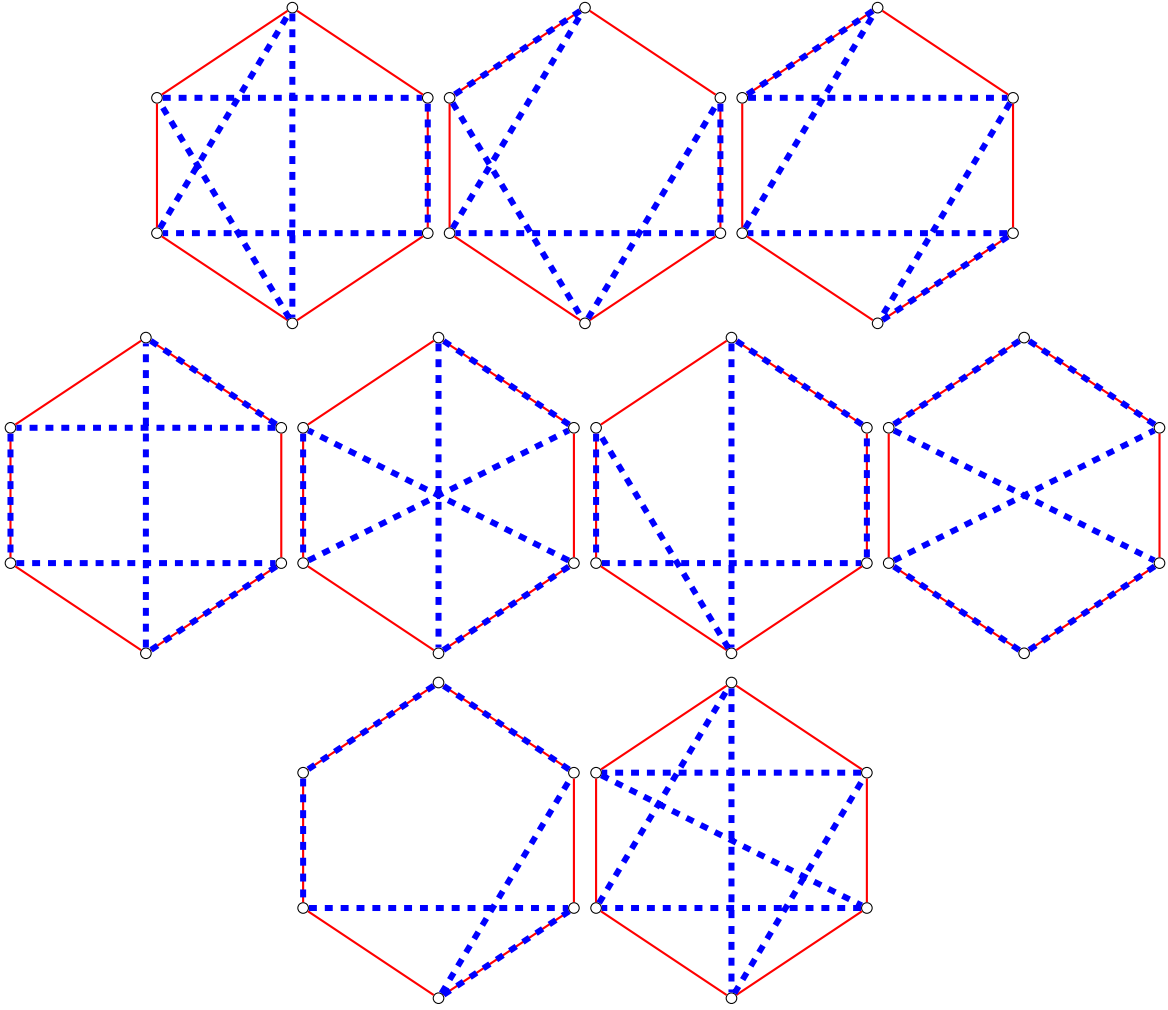
\end{proof}

\begin{lem}\label{lem:tsp3}
For $n\geq7$ we have $\cd (P_{\ts}(n))=1$.
\end{lem}
\begin{proof}
Consider two different Hamiltonian cycles $T_1$, $T_2$ in $K_n$, $n\geq 7$. For the sake of contradiction let us assume that $\bm c:=\chi(T_1)-\chi(T_2)$ is not a circuit for the Traveling Salesman polytope $P_{\ts}(n)$. Thus there exists some non-zero $\bm y$, which is not a scaling of $\bm c$, satisfying $\supp(D\bm{y}) \subseteq \supp(D\bm{c})$, where $D$ denotes the matrix of the linear constraints~\eqref{eq:tsp_description_5} and~\eqref{eq:comb}, since the linear inequalities in~\eqref{eq:tsp_description_5} and~\eqref{eq:comb} define facets for $P_{\ts}(n)$, $n\ge 7$.

\emph{Case 1: $T_1$ and $T_2$ are not disjoint.}

First, let us prove that $\bm c$ is a circuit when $T_1\cap T_2\neq \varnothing$. Then, there are two different nodes $u$ and $v$ such that $|\{e\in E\,:\, \bm{c}_e\neq 0,\, e\in \delta(u)\}|=|\{e\in E\,:\, \bm{c}_e\neq 0,\, e\in \delta(v)\}|=2$ and $\bm{c}_{uv}=0$. 

\begin{claim}\label{claim:tsp_neighborhood_equal}
Let $w$ be such that $|\{e\in E\,:\, \bm{c}_e\neq 0,\, e\in \delta(w)\}|=4$, and let the edges  $e,g\in\delta(w)$ be such that $\bm{c}_e=\bm{c}_g$. Then  $\bm{y}_e=\bm{y}_g$ holds.
\end{claim}
\begin{proof}
 For the values $\bm{c}_{uw}$ and $\bm{c}_{vw}$, we have (up to symmetry) four possibilities:
\begin{enumerate}[(i)]
	\item \label{case:tsp one minus one} $\bm{c}_{uw}=1$ and $\bm{c}_{vw}=-1$
	\item \label{case:tsp one one} $\bm{c}_{uw}=1$ and $\bm{c}_{vw}=1$
	\item \label{case:tsp zero one} $\bm{c}_{uw}=0$ and $\bm{c}_{vw}=1$
	\item \label{case:tsp zero zero}  $\bm{c}_{uw}=0$ and $\bm{c}_{vw}=0$\,.
\end{enumerate}

\bigskip
 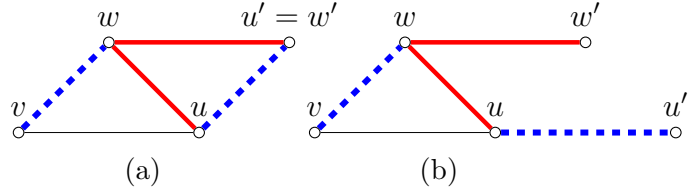
\begin{figure}[!ht]
	\begin{center}
		\begin{subfigure}{.23\linewidth}
			\begin{tikzpicture}[scale=1.2]
				\tikzset{
					graph node/.style={shape=circle,draw=black,inner sep=0pt, minimum size=4pt
      							  }
						}
				\tikzset{	plus edge/.style={red, , line width=.6mm
      							  }
						}		
				\tikzset{	minus edge/.style={blue, line width=.8mm, dashed
      							  }
						}	
				\tikzset{	no edge/.style={very thin
      							  }
						}		
				\node[graph node, label=$v$] (v) at (-1,0){};
				\node[graph node, label=$u$] (u) at (1,0){};
				\node[graph node, label=$w$] (w) at (0,1){};
				\node[graph node, label={$u'=w'$}] (u') at (2,1){};

				\draw[plus edge] (u)--(w);
				\draw[minus edge] (v)--(w);
				\draw[plus edge] (w)--(u');
				\draw[minus edge] (u)--(u');
				\draw[no edge] (u)--(v);
			\end{tikzpicture}
			\subcaption{}
			\label{case:tsp one minus one a}
		\end{subfigure}
		\begin{subfigure}{.23\linewidth}
			\begin{tikzpicture}[scale=1.2]	
				\tikzset{
					graph node/.style={shape=circle,draw=black,inner sep=0pt, minimum size=4pt
      							  }
						}
				\tikzset{	plus edge/.style={red, , line width=.6mm
      							  }
						}		
				\tikzset{	minus edge/.style={blue, line width=.8mm, dashed
	      						  }
						}	
				\tikzset{	no edge/.style={very thin
	      						  }
						}		
				\node[graph node, label=$v$] (v) at (-1,0){};
				\node[graph node, label=$u$] (u) at (1,0){};
				\node[graph node, label=$w$] (w) at (0,1){};
				\node[graph node, label=$w'$] (w') at (2,1){};
				\node[graph node, label=$u'$] (u') at (3,0){};
	
				\draw[plus edge] (u)--(w);
				\draw[minus edge] (v)--(w);
				\draw[plus edge] (w)--(w');
				\draw[minus edge] (u)--(u');
				\draw[no edge] (u)--(v);
			\end{tikzpicture}
			\subcaption{}
			\label{case:tsp one minus one b}
		\end{subfigure}
	\end{center}
	\caption[Case 1~\eqref{case:tsp one minus one} of Lemma~\ref{lem:tsp3}]{Case 1~\eqref{case:tsp one minus one}. The vector $\bm c$ has value $-1$ for blue dashed edges, $1$ for red thick edges and $0$ for thin edges. (The value of not depicted edges is not relevant for the proof.) }
	\label{fig:one minus one case}
\end{figure}

\emph{Case \eqref{case:tsp one minus one}.}
Let $u'$ be the node such that $\bm{c}_{uu'}=-1$; and $w'$ be the node such that $\bm{c}_{ww'}=1$ and $u\neq w'$. There are two possible cases: $u'=w'$ (see Figure~\ref{case:tsp one minus one a}) and $u'\neq w'$ (see Figure~\ref{case:tsp one minus one b}). In the first case (see Figure~\ref{case:tsp one minus one a}), the statement of the Claim follows by considering $\bm y(\delta(u))$, $\bm y(\delta(w))$, $\bm y(E[\{u,v,w\}])$ and $\bm y_{ww'}+\bm y_{uw'}+ \bm y_{wu}+\bm y_{wv}+\bm y_{ut}+\bm y_{w's}$, where $\bm c_{ut}=0$, $\bm c_{w's}=0$, $s\neq t$ and $s,t$ are different from $u,v,w,w'$. (Note that such $s$, $t$ exist since there are at least $3$ nodes in $K_n$ different from $u,v,w,w'$, because $n\geq 7$. For at most $2$ nodes $r$ of these $3$ nodes, we have $\bm c_{w'r}\neq 0$. For every node $r$ of these $3$ nodes, we have $\bm c_{ur}=0$.)

In the second case (see Figure~\ref{case:tsp one minus one b}), the statement of the Claim follows by considering $\bm y(\delta(u))$, $\bm y(\delta(w))$, $\bm y(E[\{u,v,w\}])$ and $\bm y_{wu}+\bm y_{uv}+ \bm y_{wv}+\bm y_{ww'}+\bm y_{uu'}+\bm y_{vs}$, where $\bm c_{vs}=0$ and $s$ is different from $u,v,w,u',w'$. (Note that such $s$ exists since there are at least $2$ nodes in $K_n$ different from $u,v,w,u',w'$, because $n\geq 7$. For at most $1$ node $r$ of these $2$ nodes, we have $\bm c_{vr}\neq 0$.)

 \begin{figure}[!ht]
	\begin{center}
		\begin{subfigure}{.23\linewidth}
			\begin{tikzpicture}[scale=1.2]
				\tikzset{
					graph node/.style={shape=circle,draw=black,inner sep=0pt, minimum size=4pt
      							  }
						}
				\tikzset{	plus edge/.style={red, , line width=.6mm
      							  }
						}		
				\tikzset{	minus edge/.style={blue, line width=.8mm, dashed
      							  }
						}	
				\tikzset{	no edge/.style={very thin
      							  }
						}		
				\node[graph node, label=$v$] (v) at (-1,0){};
				\node[graph node, label=$u$] (u) at (1,0){};
				\node[graph node, label=$w$] (w) at (0,1){};
				\node[graph node, label={$u'=w'$}] (u') at (2,1){};

				\draw[plus edge] (u)--(w);
				\draw[plus edge] (v)--(w);
				\draw[minus edge] (w)--(u');
				\draw[minus edge] (u)--(u');
				\draw[no edge] (u)--(v);
				\draw[no edge] (u')--(v);
			\end{tikzpicture}
			\subcaption{}
			\label{case:tsp one one a}
		\end{subfigure}
		\begin{subfigure}{.23\linewidth}
			\begin{tikzpicture}[scale=1.2]	
				\tikzset{
					graph node/.style={shape=circle,draw=black,inner sep=0pt, minimum size=4pt
      							  }
						}
				\tikzset{	plus edge/.style={red, , line width=.6mm
      							  }
						}		
				\tikzset{	minus edge/.style={blue, line width=.8mm, dashed
	      						  }
						}	
				\tikzset{	no edge/.style={very thin
	      						  }
						}		
				\node[graph node, label=$v$] (v) at (-1,0){};
				\node[graph node, label=$u$] (u) at (1,0){};
				\node[graph node, label=$w$] (w) at (0,1){};
				\node[graph node, label=$w'$] (w') at (2,1){};
				\node[graph node, label=$u'$] (u') at (3,0){};
	
				\draw[plus edge] (u)--(w);
				\draw[plus edge] (v)--(w);
				\draw[minus edge] (w)--(w');
				\draw[minus edge] (u)--(u');
				\draw[no edge] (u)--(v);
				\draw[no edge] (u)--(w');
			\end{tikzpicture}
			\subcaption{}
			\label{case:tsp one one b}
		\end{subfigure}
	\end{center}
	\caption[Case 1~\eqref{case:tsp one one} of Lemma~\ref{lem:tsp3}]{Case 1~\eqref{case:tsp one one}. The vector $\bm c$ has value $-1$ for blue dashed edges, $1$ for red thick edges and $0$ for thin edges. (The value of not depicted edges is not relevant for the proof.) }
	\label{fig:one one case}
\end{figure}
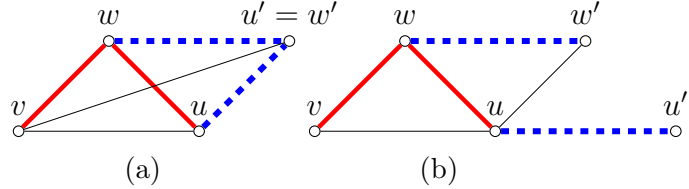

\smallskip
\emph{Case \eqref{case:tsp one one}.}
Let $u'$ be the node such that $\bm{c}_{uu'}=-1$; and $w'$ be a node such that $\bm{c}_{ww'}=-1$. There are two possible cases: $u'=w'$ (see Figure~\ref{case:tsp one one a}) and $u'\neq w'$ (see Figure~\ref{case:tsp one one b}). In the first case (see Figure~\ref{case:tsp one one a}), the statement of the Claim follows by considering $\bm y(\delta(u))$, $\bm y(\delta(v))$, $\bm y(\delta(w))$, $\bm y(E[\{v,w,w'\}])$ and $\bm y_{wu}+\bm y_{uv}+ \bm y_{vw}+\bm y_{ww'}+\bm y_{ut}+\bm y_{vs}$, where $\bm c_{ut}=0$, $\bm c_{vs}=-1$, $s\neq t$ and $s,t$ are different from $u,v,w,w'$. (Note that such $s$, $t$ trivially exist. The node $s$ is uniquely defined, and for  every node $t$ different from $u,v,w,w',s$ we have  $\bm c_{ut}=0$.)

In the second case (see Figure~\ref{case:tsp one one b}), the statement of the Claim follows by considering $\bm y(\delta(u))$, $\bm y(\delta(w))$, $\bm y(E[\{u,w,w'\}])$ and $\bm y_{wu}+\bm y_{uv}+ \bm y_{wv}+\bm y_{ww'}+\bm y_{uu'}+\bm y_{vs}$, where $\bm c_{vs}=0$ and $s$ is different from $u,v,w,u',w'$. (Note that such $s$ exists since there are at least $2$ nodes in $K_n$ different from $u,v,w,u',w'$, because $n\geq 7$. For at most $1$ node $r$ of these $2$ nodes, we have $\bm c_{vr}\neq 0$.)

 \begin{figure}[!ht]
	\begin{center}
			\begin{tikzpicture}[scale=1.2]
				\tikzset{
					graph node/.style={shape=circle,draw=black,inner sep=0pt, minimum size=4pt
      							  }
						}
				\tikzset{	plus edge/.style={red, , line width=.6mm
      							  }
						}		
				\tikzset{	minus edge/.style={blue, line width=.8mm, dashed
      							  }
						}	
				\tikzset{	no edge/.style={very thin
      							  }
						}		
				\node[graph node, label=$v$] (v) at (-1,0){};
				\node[graph node, label=$u$] (u) at (1,0){};
				\node[graph node, label={below: $w$}] (w) at (0,1){};
				\node[graph node, label={$w'$}] (w') at (1.5,1.5){};
				\node[graph node, label={$w''$}] (w'') at (-1.5,1.5){};

				\draw[no edge] (u)--(w);
				\draw[plus edge] (v)--(w);
				\draw[minus edge] (w)--(w'');
				\draw[minus edge] (w)--(w');
				\draw[no edge] (u)--(v);
			\end{tikzpicture}
	\end{center}
	\caption[Case 1~\eqref{case:tsp zero one} of Lemma~\ref{lem:tsp3}]{Case 1~\eqref{case:tsp zero one}. The vector $\bm c$ has value $-1$ for blue dashed edges, $1$ for red thick edges and $0$ for thin edges. (The value of not depicted edges is not relevant for the proof.) }
	\label{fig:zero one case}
\end{figure}
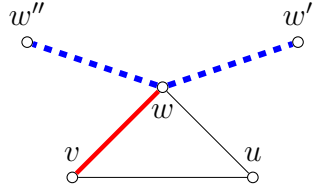

\smallskip
\emph{Case \eqref{case:tsp zero one}}
Let $w'$, $w''$ be two different nodes such that $\bm{c}_{ww'}=-1$ and $\bm{c}_{ww''}=-1$(see Figure~\ref{fig:zero one case}). The statement of the Claim follows by considering $\bm y(\delta(w))$ and $\bm y_{wu}+\bm y_{uv}+ \bm y_{vw}+\bm y_{w\bar{w}}+\bm y_{ut}+\bm y_{vs}$ for each $\bar{w} \in \{w',w''\}$, where $\bm c_{ut}=0$, $\bm c_{vs}=0$, $s\neq t$ and $s,t$ are different from $u,v,w,\bar{w}$. (Note that such $s$ and $t$ exist. Indeed, there are at least $3$ nodes in $K_n$ different from $u,v,w,\bar{w}$, because $n\geq 7$. For at most $2$ nodes $r$ of these $3$ nodes, we have $\bm c_{ur}\neq 0$.  For at most $1$ node $r$ of these $3$ nodes, we have $\bm c_{vr}\neq 0$. )

 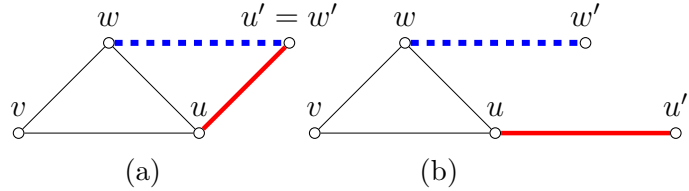
\begin{figure}[!ht]
	\begin{center}
		\begin{subfigure}{.23\linewidth}
			\begin{tikzpicture}[scale=1.2]
				\tikzset{
					graph node/.style={shape=circle,draw=black,inner sep=0pt, minimum size=4pt
      							  }
						}
				\tikzset{	plus edge/.style={red, , line width=.6mm
      							  }
						}		
				\tikzset{	minus edge/.style={blue, line width=.8mm, dashed
      							  }
						}	
				\tikzset{	no edge/.style={very thin
      							  }
						}		
				\node[graph node, label=$v$] (v) at (-1,0){};
				\node[graph node, label=$u$] (u) at (1,0){};
				\node[graph node, label=$w$] (w) at (0,1){};
				\node[graph node, label={$u'=w'$}] (u') at (2,1){};

				\draw[no edge] (u)--(w);
				\draw[no edge] (v)--(w);
				\draw[minus edge] (w)--(u');
				\draw[plus edge] (u)--(u');
				\draw[no edge] (u)--(v);
			\end{tikzpicture}
			\subcaption{}
			\label{case:tsp zero zero a}
		\end{subfigure}
		\begin{subfigure}{.23\linewidth}
			\begin{tikzpicture}[scale=1.2]	
				\tikzset{
					graph node/.style={shape=circle,draw=black,inner sep=0pt, minimum size=4pt
      							  }
						}
				\tikzset{	plus edge/.style={red, , line width=.6mm
      							  }
						}		
				\tikzset{	minus edge/.style={blue, line width=.8mm, dashed
	      						  }
						}	
				\tikzset{	no edge/.style={very thin
	      						  }
						}		
				\node[graph node, label=$v$] (v) at (-1,0){};
				\node[graph node, label=$u$] (u) at (1,0){};
				\node[graph node, label=$w$] (w) at (0,1){};
				\node[graph node, label=$w'$] (w') at (2,1){};
				\node[graph node, label=$u'$] (u') at (3,0){};
	
				\draw[no edge] (u)--(w);
				\draw[no edge] (v)--(w);
				\draw[minus edge] (w)--(w');
				\draw[plus edge] (u)--(u');
				\draw[no edge] (u)--(v);
			\end{tikzpicture}
			\subcaption{}
			\label{case:tsp zero zero b}
		\end{subfigure}
	\end{center}
	\caption[Case 1~\eqref{case:tsp zero zero} of Lemma~\ref{lem:tsp3}]{Case 1~\eqref{case:tsp zero zero}. The vector $\bm c$ has value $-1$ for blue dashed edges, $1$ for red thick edges and $0$ for thin edges. (The value of not depicted edges is not relevant for the proof.) }
	\label{fig:zero zero case}
\end{figure}
\smallskip
\emph{Case \eqref{case:tsp zero zero}}. Consider a node $w'$ and a node $u'$ such that $\bm{c}_{ww'}=-\bm{c}_{uu'}$. To prove the Claim, it is enough to show that $\bm{y}_{ww'}=-\bm{y}_{uu'}$.

There are two possible cases: $u'=w'$ (see Figure~\ref{case:tsp zero zero a}) and $u'\neq w'$ (see Figure~\ref{case:tsp zero zero b}). In Figure~\ref{fig:zero zero case}, without loss of generality we assumed that $\bm{c}_{ww'}=-1$ and $\bm{c}_{uu'}=1$.) In the first case (see Figure~\ref{case:tsp zero zero a}), we can consider $\bm y (E[\{w,u,u'\}])$ to establish $\bm{y}_{ww'}=-\bm{y}_{uu'}$. 

In the second case (see Figure~\ref{case:tsp zero zero b}), to establish $\bm{y}_{ww'}=-\bm{y}_{uu'}$ we can consider $\bm y_{wu}+\bm y_{uv}+ \bm y_{vw}+\bm y_{ww'}+\bm y_{uu'}+\bm y_{vs}$ where $\bm c_{vs}=0$ and $s$ is different from $u,v,w,u',w'$. Such $s$ exists unless $n=7$ and we have the situations in Figure~\ref{fig:zero zero case special}. (Note that otherwise such $s$ exists. Indeed, there are at least $3$ nodes in $K_n$ different from $u,v,w,u',w'$, if $n\geq 8$. For at most $2$ nodes $r$ of these $3$ nodes we have $\bm c_{vr}\neq 0$.)

Now in the case in Figure~\ref{fig:zero zero case special} and $n=7$, it is straightforward to establish that there are at least two nodes $r$ such that  $|\{e\in E\,:\, \bm{c}_e\neq 0,\, e\in \delta(r)\}|=4$. Moreover, if $|\{e\in E\,:\, \bm{c}_e\neq 0,\, e\in \delta(w')\}|=4$ then there are at least four nodes $r$ such that  $|\{e\in E\,:\, \bm{c}_e\neq 0,\, e\in \delta(r)\}|=4$. Now it is not difficult to use already considered cases \eqref{case:tsp one minus one}, \eqref{case:tsp one one}, \eqref{case:tsp zero one}, to establish $\bm{y}_{ww'}=-\bm{y}_{uu'}$.

 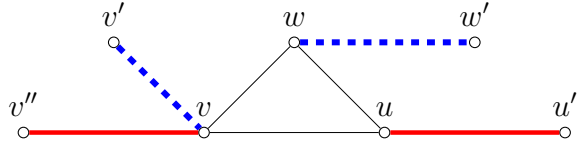
\begin{figure}[!ht]
	\begin{center}
			\begin{tikzpicture}[scale=1.2]	
				\tikzset{
					graph node/.style={shape=circle,draw=black,inner sep=0pt, minimum size=4pt
      							  }
						}
				\tikzset{	plus edge/.style={red, , line width=.6mm
      							  }
						}		
				\tikzset{	minus edge/.style={blue, line width=.8mm, dashed
	      						  }
						}	
				\tikzset{	no edge/.style={very thin
	      						  }
						}		
				\node[graph node, label=$v$] (v) at (-1,0){};
				\node[graph node, label=$v'$] (v') at (-2,1){};
				\node[graph node, label=$v''$] (v'') at (-3,0){};
				\node[graph node, label=$u$] (u) at (1,0){};
				\node[graph node, label=$w$] (w) at (0,1){};
				\node[graph node, label=$w'$] (w') at (2,1){};
				\node[graph node, label=$u'$] (u') at (3,0){};

				\draw[no edge] (u)--(w);
				\draw[no edge] (v)--(w);
				\draw[minus edge] (w)--(w');
				\draw[plus edge] (u)--(u');
				\draw[no edge] (u)--(v);
				\draw[minus edge] (v')--(v);
				\draw[plus edge] (v'')--(v);
			\end{tikzpicture}
	\end{center}
	\caption[Case 1~\eqref{case:tsp zero zero} of Lemma~\ref{lem:tsp3} (Special Case)]{Case 1~\eqref{case:tsp zero zero} (Special Case). The vector $\bm c$ has value $-1$ for blue dashed edges, $1$ for red thick edges and $0$ for thin edges. (The value of not depicted edges is not relevant for the proof.) }
	\label{fig:zero zero case special}
\end{figure}
\end{proof}

Using the above Claim for all nodes of degree 4 in a same connected component $C$ of $T_1\triangle T_2$, we establish that $\bm{y}_e=\bm{y}_g$ whenever $\bm{c}_e=\bm{c}_g$ and $e,g$ are both in $C$. On the other side, we have $\bm{y}(\delta(v))=0$ for all nodes $v$. Hence, we also have $\bm{y}_e=-\bm{y}_g$ whenever $\bm{c}_e=-\bm{c}_g$ and $e,g$ are both in $C$. 

Moreover, $\bm{y}_e=-\bm{y}_g$ holds for all edges $e$, $g$ such that $\bm{c}_e=-\bm{c}_g$. Indeed, let  $e=vv'$ and $g=uu'$ be two edges from different connected components of $T_1\triangle T_2$ such that $\bm{c}_e=-\bm{c}_g$. Consider the constraint $x(E[\{v,v',u,u'\}])\leq 3$ from~\eqref{eq:tsp_description_5}. Since $\bm{c}(E[\{v,v',u,u'\}])=0$, we have $\bm{y}(E[\{v,v',u,u'\}])=\bm{y}_e+\bm{y}_g=0$, implying $\bm{y}_e=-\bm{y}_g$. 

Hence, for $n\ge 7$ we proved that $\chi(T_1)-\chi(T_2)$ is a circuit whenever $T_1\cap T_2$ is not empty.

\bigskip

\emph{Case 2: $T_1$ and $T_2$ are disjoint.}
Let us prove that for $n\ge 7$, $\chi(T_1)-\chi(T_2)$  is a circuit whenever $T_1\cap T_2=\varnothing$.

For $n=7$ we have (up to symmetry) three possibilities for two different Hamiltonian cycles $T_1$ and $T_2$ without a common edge (see Figure~\ref{fig:tsp7}). In all these cases $\chi(T_1)-\chi(T_2)$ is a circuit.
 \begin{figure}
	\begin{center}
		\begin{tikzpicture}
	
			\tikzset{
				graph node/.style={shape=circle,draw=black,inner sep=0pt, minimum size=4pt
      						  }
					}
			\tikzset{	plus edge/.style={red, thick
      						  }
					}		
			\tikzset{			minus edge/.style={blue, line width=.8mm, dashed
      						  }
					}		
			\def\rad{2}	
				\foreach \i in {1,...,7}{
					\node[graph node] (v\i) at ({\rad*cos((2*\i)*360/14+90)},{\rad*sin((2*\i)*360/14+90)}){};}

			\draw[plus edge] (v1)--(v2)--(v3)--(v4)--(v5)--(v6)--(v7)--(v1);
			\draw[minus edge] (v1)--(v3)--(v5)--(v7)--(v2)--(v4)--(v6)--(v1);
		\end{tikzpicture}
				\begin{tikzpicture}
	
			\tikzset{
				graph node/.style={shape=circle,draw=black,inner sep=0pt, minimum size=4pt
      						  }
					}
			\tikzset{	plus edge/.style={red, thick
      						  }
					}		
			\tikzset{			minus edge/.style={blue, line width=.8mm, dashed
      						  }
					}		
			\def\rad{2}	
				\foreach \i in {1,...,7}{
					\node[graph node] (v\i) at ({\rad*cos((2*\i)*360/14+90)},{\rad*sin((2*\i)*360/14+90)}){};}

			\draw[plus edge] (v1)--(v2)--(v3)--(v4)--(v5)--(v6)--(v7)--(v1);
			\draw[minus edge] (v1)--(v3)--(v5)--(v7)--(v2)--(v6)--(v4)--(v1);
		\end{tikzpicture}
		\begin{tikzpicture}
	
			\tikzset{
				graph node/.style={shape=circle,draw=black,inner sep=0pt, minimum size=4pt
      						  }
					}
			\tikzset{	plus edge/.style={red, thick
      						  }
					}		
			\tikzset{			minus edge/.style={blue, line width=.8mm, dashed
      						  }
					}		
			\def\rad{2}	
				\foreach \i in {1,...,7}{
					\node[graph node] (v\i) at ({\rad*cos((2*\i)*360/14+90)}, 	{\rad*sin((2*\i)*360/14+90)}){};}

			\draw[plus edge] (v1)--(v2)--(v3)--(v4)--(v5)--(v6)--(v7)--(v1);
			\draw[minus edge] (v7)--(v3)--(v6)--(v1)--(v4)--(v2)--(v5)--(v7);
		\end{tikzpicture}
	\end{center}
	\caption[All possible cases for two different Hamiltonian cycles in $K_7$ without a common edge]{All possible cases (up to symmetry) for two different Hamiltonian cycles $T_1$ and $T_2$ in $K_7$ without a common edge. Here, the edges of $T_2$ are depicted as dashed edges.}
	\label{fig:tsp7}
\end{figure}
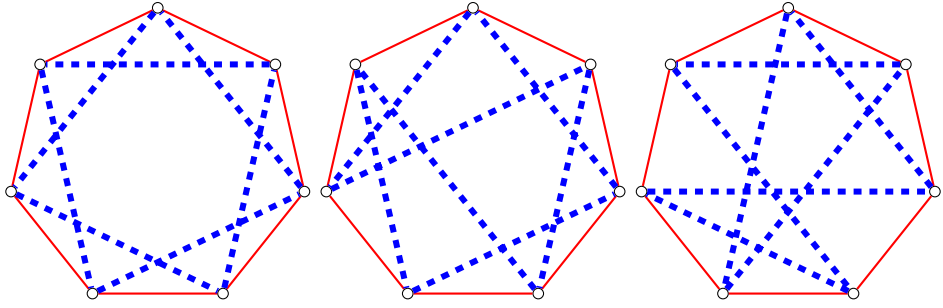

For $n\geq 8$ let us show the following Claim.
\begin{claim}\label{claim:tsp_edges_equal}
Let $w$ be a node and $e,g\in\delta(w)$ be such that $\bm{c}_e=\bm{c}_g$. Then  $\bm{y}_e=\bm{y}_g$ holds.
\end{claim}
\begin{proof}
Let $e$, $g$ be $wv$, $wu$ for some two nodes $u$, $v$. We may assume that $u$ and $v$ are different, since otherwise the statement of the Claim is trivial. 

Without loss of generality, we may assume $\bm{c}_e=1$ and $\bm{c}_g=1$. There are two possible cases
\begin{enumerate}[(a)]
	\item \label{case:tsp one one zero} $\bm{c}_{uv}=-1$
	\item \label{case:tsp one one one} $\bm{c}_{uv}=0$.
\end{enumerate}

In the case~\eqref{case:tsp one one zero}, let $w'$, $w''$ be two different nodes such that $\bm{c}_{ww'}=-1$ and $\bm{c}_{ww''}=-1$. For each $\bar{w}\in \{w',w''\}$, to establish $\bm{y}_{w\bar{w}}=\bm{y}_{uv}$ consider $\bm{y}_{wu}+\bm{y}_{uv}+\bm{y}_{wv}+\bm{y}_{w\bar{w}}+\bm{y}_{vs}+\bm{y}_{ut}$, where $s$, $t$, $s\neq t$ are two nodes different from $u$, $v$, $w$, $\bar w$ such that $\bm{c}_{vs}=0$ and $\bm{c}_{ut}=0$. (Note that such nodes $s$, $t$ exist. Indeed, since $n\geq 8$ there are at least $4$ nodes in $K_n$ different from $u$, $v$, $w$, $\bar w$. There are at most $2$ nodes $r$ of these $4$ nodes such that $\bm{c}_{vr}\neq 0$. Also there are at most $2$ nodes $r$ of these $4$ nodes such that $\bm{c}_{ur}\neq 0$.) To establish $\bm{y}_e=\bm{y}_g$, now it is enough to consider $\bm{y}(E[\{s,w^\star,w\}])$ and $\bm{y}(\delta(w))$, where $s\in \{u,v\}$, $w^\star\in \{w',w''\}$ such that $\bm{c}_{sw^\star}=0$. (Note that such nodes $s$, $w^\star$ exist, since otherwise $T_2$ has a subtour.).

In the case~\eqref{case:tsp one one one}, let $u'$, $u''$ be two different nodes such that $\bm{c}_{uu'}=\bm{c}_{uu''}=-1$, and let $v'$, $v''$ be two different nodes such that $\bm{c}_{vv'}=\bm{c}_{vv''}=-1$ (see Figure~\eqref{fig:one one one case}). First note that $\{u',u''\}\neq\{v',v''\}$ as otherwise $T_2$ contains a subtour.  Then we may assume that $v'\notin\{u',u''\}$ and $u'\notin\{v',v''\}$. (Note that $v''$ could be equal to $u''$).

It follows that $\bm y_{uu'}=\bm y_{uu''}$ by considering $\bm{y}_{wu}+\bm{y}_{uv}+\bm{y}_{vw}+\bm{y}_{vv'}+\bm{y}_{u\bar{u}}+\bm{y}_{wz}$ for each $\bar{u}\in\{u',u''\}$, where $\bm c_{wz}=0$ and $z$ is different from $u,v,w,v',$ and $\bar{u}$.  (Note that such a $z$ exists.  Indeed there are at least $3$ nodes in $K_n$ different from $u,v,w,v'$ and $\bar{u}$ if $n\geq8$.  For at most $2$ nodes $r$ of these $3$ nodes, we have $\bm c_{wr}\neq0$.).  By symmetry, we also have that $\bm y_{vv'}=\bm y_{vv''}$.

There exists $\bar{u}\in\{u',u''\}$ such that $\bm c_{w\bar{u}}=0$ as otherwise $T_2$ contains a subtour.  Then it follows that $\bm y_{uw}=-\bm y_{u\bar{u}}$ by considering $\bm y(E[\{w,u,\bar u\}])$.  Therefore, $\bm y_{uw}=-\bm y_{uu'}$ and $\bm y_{uw}=-\bm y_{uu''}$.  Similarly, $\bm y_{vw}=-\bm y_{vv'}$ and $\bm y_{vw}=-\bm y_{vv''}$.

 \begin{figure}[!ht]
	\begin{center}
			\begin{tikzpicture}[scale=1.2]	
				\tikzset{
					graph node/.style={shape=circle,draw=black,inner sep=0pt, minimum size=4pt
      							  }
						}
				\tikzset{	plus edge/.style={red, , line width=.6mm
      							  }
						}		
				\tikzset{	minus edge/.style={blue, line width=.8mm, dashed
	      						  }
						}	
				\tikzset{	no edge/.style={very thin
	      						  }
						}		
				\node[graph node, label=$v$] (v) at (-1,0){};
				\node[graph node, label=$v'$] (v') at (-2,1){};
				\node[graph node, label=$v''$] (v'') at (-3,0){};
				\node[graph node, label=$u$] (u) at (1,0){};
				\node[graph node, label=$w$] (w) at (0,1){};
				\node[graph node, label=$u'$] (u') at (2,1){};
				\node[graph node, label=$u''$] (u'') at (3,0){};

				\draw[plus edge] (u)--(w);
				\draw[plus edge] (v)--(w);
				\draw[no edge] (u)--(v);
				\draw[minus edge] (u)--(u');
				\draw[minus edge] (u)--(u'');
				\draw[minus edge] (v)--(v');
				\draw[minus edge] (v)--(v'');
			\end{tikzpicture}
	\end{center}
	\caption[Case 2~\eqref{case:tsp one one one} of Lemma~\ref{lem:tsp3}]{Case 2~\eqref{case:tsp one one one} of Lemma~\ref{lem:tsp3}. The vector $\bm c$ has value $-1$ for blue dashed edges, $1$ for red thick edges and $0$ for thin edges. (The value of not depicted edges is not relevant for the proof.) }
	\label{fig:one one one case}
\end{figure}
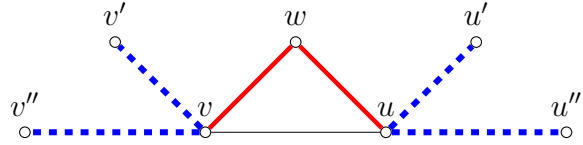

Now, if $\bm c_{vu'}\neq0$, then since $u'\notin\{v',v''\}$, we have that $\bm c_{vu'}=1$.  Then it follows that $\bm y_{vu'}=\bm y_{vw}$ by considering $\bm y(\delta(v))$.  It follows that $\bm y_{vu'}=-\bm y_{uu'}$ by considering $\bm y(E[\{u,u',v\}])$.  Then in this case we have that
$$
\bm y_{uw}=-\bm y_{uu'}=\bm y_{vu'}=\bm y_{vw},
$$
and therefore $\bm y_g=\bm y_e$, as desired.

Otherwise, $\bm c_{vu'}=0$, and by symmetry we may assume that $\bm c_{uv'}=0$ as well. There exists a node $v'''\neq u'$ such that $\bm c_{vv'''}=1$.  It follows that $\bm y_{vv'''}=\bm y_{vw}$ by considering $\bm y(\delta(v))$.

 \begin{figure}[!ht]
	\begin{center}
			\begin{tikzpicture}[scale=1.2]	
				\tikzset{
					graph node/.style={shape=circle,draw=black,inner sep=0pt, minimum size=4pt
      							  }
						}
				\tikzset{	plus edge/.style={red, , line width=.6mm
      							  }
						}		
				\tikzset{	minus edge/.style={blue, line width=.8mm, dashed
	      						  }
						}	
				\tikzset{	no edge/.style={very thin
	      						  }
						}		
				\node[graph node, label=$v$] (v) at (-1,0){};
				\node[graph node, label=$v'$] (v') at (-2,1){};
				\node[graph node, label=$v'''$] (v''') at (-3,0){};
				\node[graph node, label=$u$] (u) at (1,0){};
				\node[graph node, label=$w$] (w) at (0,1){};
				\node[graph node, label=$u'$] (u') at (2,1){};

				\draw[plus edge] (u)--(w);
				\draw[plus edge] (v)--(w);
				\draw[no edge] (u)--(v);
				\draw[minus edge] (u)--(u');
				\draw[plus edge] (v)--(v''');
				\draw[minus edge] (v)--(v');
				\draw[no edge] (v)--(u');
				\draw[no edge] (u)--(v');
				\draw[no edge] (w)--(u');
			\end{tikzpicture}
	\end{center}
	\caption[Case 2~\eqref{case:tsp one one one} of Lemma~\ref{lem:tsp3}: First sub-case]{Case 2~\eqref{case:tsp one one one} of Lemma~\ref{lem:tsp3}: First sub-case. The vector $\bm c$ has value $-1$ for blue dashed edges, $1$ for red thick edges and $0$ for thin edges. (The value of not depicted edges is not relevant for the proof.) }
	\label{fig:one one one case A}
\end{figure}
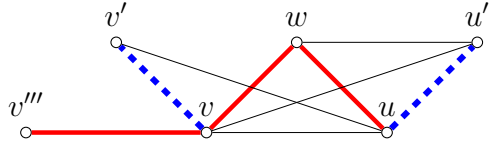

If $\bm c_{wu'}=0$  (see Figure~\eqref{fig:one one one case A}), then it follows that $\bm y_{vv'''}=-\bm y_{uu'}$ by considering $\bm{y}_{uu'}+\bm{y}_{u'v}+\bm{y}_{vu}+\bm{y}_{vv'''}+\bm{y}_{u'w}+\bm{y}_{uv'}$.  Then in this case we have that
$$
\bm y_{uw}=-\bm y_{uu'}=\bm y_{vv'''}=\bm y_{vw},
$$
and therefore $\bm y_g=\bm y_e$, as desired.

Otherwise, $\bm c_{wu'}=-1$.  Then $\bm c_{wu''}=0$, as otherwise $T_2$ contains a subtour.

 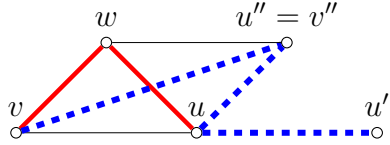
\begin{figure}[!ht]
	\begin{center}
			\begin{tikzpicture}[scale=1.2]	
				\tikzset{
					graph node/.style={shape=circle,draw=black,inner sep=0pt, minimum size=4pt
      							  }
						}
				\tikzset{	plus edge/.style={red, , line width=.6mm
      							  }
						}		
				\tikzset{	minus edge/.style={blue, line width=.8mm, dashed
	      						  }
						}	
				\tikzset{	no edge/.style={very thin
	      						  }
						}		
				\node[graph node, label=$v$] (v) at (-1,0){};
				\node[graph node, label={$u''=v''$}] (u'') at (2,1){};
				\node[graph node, label=$u$] (u) at (1,0){};
				\node[graph node, label=$w$] (w) at (0,1){};
				\node[graph node, label=$u'$] (u') at (3,0){};

				\draw[plus edge] (u)--(w);
				\draw[plus edge] (v)--(w);
				\draw[no edge] (u)--(v);
				\draw[minus edge] (u)--(u');
				\draw[minus edge] (u)--(u'');
				\draw[minus edge] (v)--(u'');
				\draw[no edge] (w)--(u'');
			\end{tikzpicture}
	\end{center}
	\caption[Case 2~\eqref{case:tsp one one one} of Lemma~\ref{lem:tsp3}: Second sub-case]{Case 2~\eqref{case:tsp one one one} of Lemma~\ref{lem:tsp3}: Second sub-case. The vector $\bm c$ has value $-1$ for blue dashed edges, $1$ for red thick edges and $0$ for thin edges. (The value of not depicted edges is not relevant for the proof.) }
	\label{fig:one one one case B}
\end{figure}

If $v''=u''$ (see Figure~\eqref{fig:one one one case B}), it follows that $\bm y_{vw}=-\bm y_{uu'}$ by considering  $\bm{y}_{uu''}+\bm{y}_{u''w}+\bm{y}_{wu}+\bm{y}_{uu'}+\bm{y}_{wv}+\bm{y}_{u''z}$, where $\bm c_{u''z}=0$ and $z$ is different from $u,u'',w,u',$ and $v$. (Note that such a $z$ exists.  Indeed there are at least $3$ nodes in $K_n$ different from $u,u'',w,u',$ and $v$ if $n\geq8$.  For at most $2$ nodes $r$ of these $3$ nodes, we have $\bm c_{u''r}\neq0$.).  Then in this case we have that
$$
\bm y_{uw}=-\bm y_{uu'}=\bm y_{vw}
$$
and therefore, $\bm y_g=\bm y_e$, as desired.

 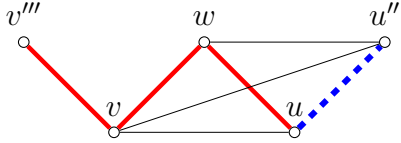
\begin{figure}[!ht]
	\begin{center}
			\begin{tikzpicture}[scale=1.2]	
				\tikzset{
					graph node/.style={shape=circle,draw=black,inner sep=0pt, minimum size=4pt
      							  }
						}
				\tikzset{	plus edge/.style={red, , line width=.6mm
      							  }
						}		
				\tikzset{	minus edge/.style={blue, line width=.8mm, dashed
	      						  }
						}	
				\tikzset{	no edge/.style={very thin
	      						  }
						}		
				\node[graph node, label=$v$] (v) at (-1,0){};
				\node[graph node, label=$u''$] (u'') at (2,1){};
				\node[graph node, label=$u$] (u) at (1,0){};
				\node[graph node, label=$w$] (w) at (0,1){};
				\node[graph node, label=:$v'''$] (v''') at (-2,1){};

				\draw[plus edge] (u)--(w);
				\draw[plus edge] (v)--(w);
				\draw[no edge] (u)--(v);
				\draw[minus edge] (u)--(u'');
				\draw[plus edge] (v)--(v''');
				\draw[no edge] (v)--(u'');
				\draw[no edge] (w)--(u'');
			\end{tikzpicture}
	\end{center}
	\caption[Case 2~\eqref{case:tsp one one one} of Lemma~\ref{lem:tsp3}: Third sub-case]{Case 2~\eqref{case:tsp one one one} of Lemma~\ref{lem:tsp3}: Third sub-case. The vector $\bm c$ has value $-1$ for blue dashed edges, $1$ for red thick edges and $0$ for thin edges. (The value of not depicted edges is not relevant for the proof.) }
	\label{fig:one one one case C}
\end{figure}

Otherwise, $v''\neq u''$. 

If $\bm c_{vu''}=0$ (see Figure~\eqref{fig:one one one case C}), it follows that $\bm y_{vv'''}=-\bm y_{uu''}$ by considering $\bm{y}_{u''v}+\bm{y}_{vu}+\bm{y}_{uu''}+\bm{y}_{u''w}+\bm{y}_{vv'''}+\bm y_{uz}$, where $\bm c_{uz}=0$ and $z$ is different from $u'',v,u,w,$ and $v'''$. (Note that such a $z$ exists.  Indeed there are at least $3$ nodes in $K_n$ different from $u'',v,u,w,$ and $v'''$ if $n\geq8$.  For at most $2$ nodes $r$ of these $3$ nodes, we have $\bm c_{ur}\neq0$.).  Then in this case we have that
$$
\bm y_{uw}=-\bm y_{uu''}=\bm y_{vv'''}=\bm y_{vw},
$$
and therefore, $\bm y_g=\bm y_e$, as desired.

Finally, if instead $\bm c_{vu''}=1$ (that is, $u''=v'''$), then it follows that $\bm y_{vu''}=-\bm y_{uu''}$ by considering $\bm y(E[\{u,v,u''\}])$.  Then in this case we have that 
$$
\bm y_{uw}=-\bm y_{uu''}=\bm y_{vu''}=\bm y_{vw},
$$
and therefore $\bm y_g=\bm y_e$, as desired.
\end{proof}

Together, claims~\ref{claim:tsp_edges_equal} and~\ref{claim:tsp_neighborhood_equal} implies that, up to scaling, $\bm y = \bm c$, a contradiction.  Thus, for $n\geq 7$ and for any two different Hamiltonian cycles $T_1$, $T_2$, we have that $\bm c =\chi(T_1)-\chi(T_2)$ is a circuit for the Traveling Salesman polytope. 
\end{proof}

\begin{proof}\emph{(Proof of Theorem \ref{thm:tsp})}
The cases $n=3$ and $n=4$ are trivial. Indeed, $P_{\ts}(3)$ and $P_{\ts}(4)$ are simplices, and thus every two vertices of $P_{\ts}(3)$ and $P_{\ts}(4)$ form an edge. The cases, $n=5$, $n=6$ and $n\geq 7$ are covered by Lemma~\ref{lem:tsp1}, Lemma~\ref{lem:tsp2} and Lemma~\ref{lem:tsp3},  respectively.
\end{proof}

\bigskip 

\section{Fractional Stable Set Polytope}\label{sec:fSTAB}

Given a connected graph $G=(V,E)$ with at least two nodes, the Fractional Stable Set polytope is defined as follows
$$
	P_{\fstab}(G):=\{ \bm{x}\in \R^V\,:\, \bm{x}_u+\bm{x}_v \le 1 \text{  for all  } uv\in E,\, \bm{x}\ge \bm{0} \}\,.
$$
The Fractional Stable Set polytope is a well studied polytope. In particular, it is known that all vertices of it are half-integral~\cite{B70} i.e., $\bm{x}\in \{0,1/2,1\}^V$ whenever $\bm{x}$ is a vertex of $P_{\fstab}(G)$. In~\cite{Michini2014}, it is shown that the combinatorial diameter of $P_{\fstab}(G)$ is bounded from above by the number of nodes in $G$. 

Before we study the circuit diameter of the Fractional Stable Set polytope let us study the circuits of this polytope.  Its circuits admit a nice characterization captured by the lemma below.

\begin{lem}
\label{lem:fstab circuits}
For a graph $G=(V,E)$, a vector $\bm{c}$, $\bm{c}\neq \bm{0}$ is a circuit of $P_{\fstab}(G)$ if and only if the graph $G'$ with the node set $V':=\{v\in V\,:\, \bm{c}_v\neq 0\}$ and the edge set $E':=\{e\in E\,:\, e=uv,\, u,v \in V'\text{  and  }\bm{c}_u+\bm{c}_v=0\}$ is connected.
\end{lem}
\begin{proof}
Let $G'$ be not connected and let $C$ be a connected component of $G'$ with a node set $U$. Let us define the vector $\bm{c}'\in \R^V$ as
$$
	\bm{c}'_v:=\begin{cases}
				\bm{c}_v&\text{  if  } v\in U\\
				0 &\text{otherwise}\,.
			\end{cases}
$$
The vector $\bm{c}$ is not a circuit of $P_{\fstab}(G)$ since the vector $D \bm{c}'$ has a smaller support than $D \bm{c}$, where $D$ is the linear constraint matrix in the minimal description of  $P_{\fstab}(G)$. 

On the other hand, it is straightforward to check that if $G'$ is connected, then $\bm{c}$ is a unique (up to scaling) non zero solution of the below system
	\begin{align*}
		&\bm{y}_v=0 \quad&&\text{  for all  } v\in V \text{  such that  }\bm{c}_v=0\\
		&\bm{y}_v+\bm{y}_u=0 \quad&&\text{  for all  } uv\in E \text{  such that  }\bm{c}_v+\bm{c}_u=0\,,
	\end{align*}
showing that $\bm{c}$ is a circuit of $P_{\fstab}(G)$.
\end{proof}

\bigskip
To study the circuit diameter of the Fractional Stable Set polytope we need the following notation. For a node $v$, let $B{(v,0)}$ be defined as $\{v\}$. For integer positive $k$, we define $B(v, k)$ to be the set of nodes which are at distance at most $k$ from $v$. The set of nodes which are at distance exactly $k$ from $v$ is denoted by $N(v,k)$ i.e., $N(v, k):=B(v, k)\setminus B(v, k-1)$. The eccentricity $\varepsilon(v)$ of a node $v\in V$ is minimum $k$ such that $V=B(v, k)$.

\begin{lem}\label{lem:stab ecce}
Let $v$ be any node in a connected graph $G=(V,E)$ with at least two nodes. Then $\cd (P_{\fstab}(G))$ is $\mathcal{O}(\varepsilon(v))$.
\end{lem}

\begin{proof}

Let $\bm{x}'$ and $\bm{x}''$ be two vertices of $P_{\fstab}(G)$. Let us show that $\cdist(\bm{x}',\bm{x}'')$ is at most ~$4\varepsilon(v)+c$ for some constant $c$. To do this we construct a circuit walk from $\bm{x}'$ to $\bm{x}''$.  The walk will correspond to two different phases. In Phase I we construct a circuit walk from $\bm{x}'$ to some ``well structured" point $\bm{y}'$, and in Phase II we move from $\bm{y}'$ to $\bm{x}''$ by another circuit walk. 

To simplify the exposition, in the proof we assume that $G$ is a non-bipartite graph. It will be clear from the analysis of the length of the circuit walk that the bound in the statement of the lemma is also satisfied in the bipartite case.

\bigskip

{\bf Phase I:} Let us assume that $b$ is the smallest $k$ such that the subgraph of $G$ induced by $B(v,k)$ is non-bipartite.

\smallskip
{\bf Start of Phase I:}  If $b$ is odd, we first take a circuit walk from $\bm{x}'$ to a point $\bm{z}$ with $\bm{z}_v=0$ and $\bm{z}_u=\phi$ for $u\in N(v,1)$, where $\phi:=1/2$ if $b=1$ and  $\phi:=1$ otherwise.
If $b$ is even, we start by a circuit walk from $\bm{x}'$ to a point $\bm{z}$ with $\bm{z}_v=1$, $\bm{z}_u=0$ for $u\in N(v,1)$ and $\bm{z}_u=\phi$ for $u\in N(v,2)$, where $\phi:=1/2$ if $b=2$ and  $\phi:=1$ otherwise. Initialize $t:=1$ if $b$ is odd and $t:=2$ otherwise.

\begin{claim}
\label{claim:PhaseI one}
If at the beginning of Phase I we have $t=1$,  then $4$ circuit steps are enough to reach $\bm z$ from $\bm x'$.
\end{claim}

\begin{proof}
In the proof we are going to show that from every point $\bm x'\in \{0,1/2,1\}^V$, $\bm x'\in P_{\fstab}(G)$ we can reach a desired $\bm z$  in at most $4$ circuit steps. Note, that $\bm x'\in \{0,1/2,1\}^V$, $\bm x'\in P_{\fstab}(G)$ is a weaker assumption than the assumption that $\bm x'$ is a vertex of $P_{\fstab}(G)$. We weaken the assumptions on $\bm x'$ in the proof of this claim for the sake of exposition.

First suppose that $b=1$.  There are three possible cases:
\begin{enumerate}
    \item \label{case:fstab one one one} $\bm x'_v=1$
    \item \label{case:fstab one one half} $\bm x'_v=1/2$
    \item \label{case:fstab one one zero} $\bm x'_v=0$.
\end{enumerate}
In the case~\ref{case:fstab one one one}, we have $\bm x'_w=0$ for all $w\in N(v,1)$, since $\bm x'_v+\bm x'_w\leq1$.  Then let $\bm c$ be defined as the following vector:
\begin{equation*}
	\bm{c}_u=\begin{cases}
			-1/2    & \text{  if  }  u=v\\
			1/2    & \text{  if  } u\in N(v,1)\\
			-1/2 & \text{  if  } u\in N(v,2),\,\,\bm x'_u=1\\
			0 & \text{  else}\,.
		\end{cases}
\end{equation*}
By Lemma~\ref{lem:fstab circuits}, $\bm c$ is a circuit.  Let $\bm y=\bm x'+\bm c$. Then clearly, $\bm y$ is feasible for $P_{\fstab}(G)$.  In particular, for any edge $uw$ with $u,w\in N(v,1)$, we have that $\bm y_{u}+\bm y_{w}=1$.  Similarly, for any edge $uw$ with  $u\in N(v,1)$ and  $w\in N(v,2)$, we have that $\bm y_{u}+\bm y_{w}\leq1$. Furthermore,  $\bm y$ is one circuit step from $\bm x'$ since $b=1$ implies that there exists an edge $uw$, $u,w\in N(v,1)$. Now, let $\bm c'$ be defined as the following vector:
\begin{equation*}
	\bm{c}'_u=\begin{cases}
			-1/2    & \text{  if  }  u=v\\
			0       & \text{  else}\,.
		\end{cases}
\end{equation*}
By Lemma~\ref{lem:fstab circuits}, $\bm c'$ is a circuit. Moreover, $\bm y+\bm c'$ is the desired point $\bm z$. Note, that $\bm z$ is one circuit step from $\bm y$ as $\bm z_v=0$.  Hence, in this case $2$ circuit steps are enough to reach $\bm z$ from $\bm x'$.

In the case~\ref{case:fstab one one half}, we have $\bm x'_w\leq1/2$ for all $w\in N(v,1)$.  Then let $\bm c$ be defined as the following vector:
\begin{equation*}
	\bm{c}_u=\begin{cases}
			-1/2    & \text{  if  }  u=v\\
			1/2    & \text{  if  } u\in N(v,1),\,\, \bm x'_u=0\\
			-1/2 & \text{  if  } u\in N(v,2),\,\, \bm x'_u=1\\
			0 & \text{  else}\,.
		\end{cases}
\end{equation*}
By Lemma~\ref{lem:fstab circuits}, $\bm c$ is a circuit.  Let $\bm z=\bm x'+\bm c$. Then clearly, $\bm z$ is feasible for $P_{\fstab}(G)$.  In particular, for any edge $uw$ with $u,w\in N(v,1)$, we have that $\bm z_{u}+\bm z_{w}=1$.  Similarly, for any edge $uw$ with $u\in N(v,1)$ and $w\in N(v,2)$, we have that $\bm z_u+\bm z_w\leq1$. Furthermore,  $\bm z$ is one circuit step away from $\bm x'$ as $\bm z_v=0$. Thus $\bm z$ is the desired point, and in this case $1$ circuit step is enough to reach  $\bm z$ from $\bm x'$.

In the case~\ref{case:fstab one one zero}, let $\bm c$ be defined as the following vector:
\begin{equation*}
	\bm{c}_u=\begin{cases}
			1/2    & \text{  if  }  u=v\\
			-1/2    & \text{  if  } u\in N(v,1),\,\, \bm x'_u>0\\
			0 & \text{  else}\,.
		\end{cases}\,
\end{equation*}
By Lemma~\ref{lem:fstab circuits}, $\bm c$ is a circuit.  Let $\bm y=\bm x'+\alpha \bm c$ be the point which is one circuit step from $\bm x'$, where $\alpha\geq 0$. Clearly $\alpha \in \{0,1,2\}$. We have $\alpha\geq 1$, since $\bm x'+\bm c$ is feasible for $P_{\fstab}(G)$. Thus $\alpha \in \{1,2\}$, hence $\bm y\in \{0,1/2,1\}^V$ and $\bm y_v \in \{1/2,1\}$. Due to the considered cases~\ref{case:fstab one one one} and~\ref{case:fstab one one half}, we know that from $\bm y$ we can reach a desired $\bm z$ in at most $3$ circuit steps. Thus, a desired point $\bm z$ can be reached from $\bm x'$ in at most $4$ circuit steps.

Now, suppose $b>1$.  We have the same three cases as when $b=1$, and we will refer to them identically.

In the case~\ref{case:fstab one one one}, we have that $\bm x'_w=0$ for all $w\in N(v,1)$.  Then let $\bm c$ be defined as the following vector:
\begin{equation*}
	\bm{c}_u=\begin{cases}
			-1/2    & \text{  if  }  u=v\\
			1/2    & \text{  if  } u\in N(v,1)\\
			-1/2 & \text{  if  } u\in N(v,2),\,\, \bm x'_u>0\\
			0 & \text{  else}\,.
		\end{cases}\,
\end{equation*}
By Lemma~\ref{lem:fstab circuits}, $\bm c$ is a circuit.  Let $\bm y=\bm x'+\alpha \bm c$ be the point which is one circuit step from $\bm x'$, where $\alpha\geq 0$. Clearly $\alpha \in \{0,1,2\}$. We have $\alpha\geq 1$, since $\bm x'+\bm c$ is feasible for $P_{\fstab}(G)$. Thus $\alpha \in \{1,2\}$.

First, suppose that $\alpha=1$. Then let $\bm c'$ be the following vector:
\begin{equation*}
	\bm{c}'_u=\begin{cases}
			-1/2    & \text{  if  }  u=v\\
			1/2    & \text{  if  } u\in N(v,1)\\
			-1/2 & \text{  if  } u\in N(v,2),\,\, \bm y_u>0\\
			0 & \text{  else}\,.
		\end{cases}\,
\end{equation*}
By Lemma~\ref{lem:fstab circuits}, $\bm c'$ is a circuit.  Let $\bm z=\bm y+\bm c'$.  Then $\bm z$ is feasible and $\bm z$ is one circuit step from $\bm y$ as $\bm z_v=0$. Thus, a desired point $\bm z$ is at most $2$ circuit steps from $\bm x'$ if $\alpha=1$. Now, suppose $\alpha=2$.  Then $\bm z=x'+2 \bm c$ is a desired  point.  Thus, if $\alpha=2$ then $1$ circuit step is enough to reach $\bm z$ from $\bm x'$.

In the case~\ref{case:fstab one one half}, we have that $\bm x'_w\leq1/2$ for all $w\in N(v,1)$.  Then let $\bm c$ be defined as the following vector:
\begin{equation*}
	\bm{c}_u=\begin{cases}
			1/2    & \text{  if  }  u=v\\
			-1/2    & \text{  if  } u\in N(v,1),\,\, \bm x'_u=1/2\\
			0 & \text{  else}\,.
		\end{cases}\,
\end{equation*}
By Lemma~\ref{lem:fstab circuits}, $\bm c$ is a circuit.  Let $\bm y=\bm x'+\bm c$.  Then $\bm y$ is feasible, and is one circuit step from $\bm x'$.  Note that $\bm y$ satisfies the conditions of case~\ref{case:fstab one one one}, thus a desired point $\bm z$ can be achieved in at most $2$ circuit steps from the point $\bm y$.  Thus, $\bm z$ is at most $3$ circuit steps from $\bm x'$.

In the case~\ref{case:fstab one one zero}, let $\bm c$ be defined as the following vector:
\begin{equation*}
	\bm{c}_u=\begin{cases}
			1/2    & \text{  if  }  u=v\\
			-1/2    & \text{  if  } u\in N(v,1),\,\, \bm x'_u>0\\
			0 & \text{  else}\,.
		\end{cases}\,
\end{equation*}
By Lemma~\ref{lem:fstab circuits}, $\bm c$ is a circuit.  Let $\bm y=\bm x'+\alpha \bm c$ be the point which is one circuit step from $\bm x'$, where $\alpha\geq 0$. Clearly $\alpha \in \{0,1,2\}$. We have $\alpha\geq 1$, since $\bm x'+\bm c$ is feasible for $P_{\fstab}(G)$. Thus $\alpha \in \{1,2\}$, hence $\bm y\in \{0,1/2,1\}^V$ and $\bm y_v \in \{1/2,1\}$. Due to the considered cases~\ref{case:fstab one one one} and~\ref{case:fstab one one half}, we know that from $\bm y$ we can reach a desired $\bm z$ in at most $3$ circuit steps. Thus, a desired point $\bm z$ can be reached from $\bm x'$ in at most $4$ circuit steps.

Therefore, in all cases, we need at most $4$ circuit steps to reach $\bm z$ from $\bm x'$.
\end{proof}

The proof of the next claim is essentially identical to the proof of Claim~\ref{claim:PhaseI one}.
\begin{claim}
\label{claim:PhaseI two}
If at the beginning of Phase I we have $t=2$,  then $6$ circuit steps are enough to reach $\bm z$ from $\bm x'$.
\end{claim}

\smallskip

{\bf Invariants for $\bm{z}$ and $t$ in Phase I:}
During Phase I, we update  $\bm{z}$ and $t$ such that at each moment of time the following holds for all $u\in N(v,k)$, for all $k\leq t$:
\begin{equation}\label{eq:phaseI_invariant}\tag{$\star$}
	\bm{z}_u=\begin{cases}
			0    & \text{  if  }  k\equiv b+1\mod 2\\
			1    & \text{  if  } k < b \text{  and  } k\equiv b\mod 2\\
			1/2 & \text{  if  } k\ge b \text{  and  } k\equiv b\mod 2\,.
		\end{cases}\,
\end{equation}
By construction, $\bm{z}$ and $t$ defined at the beginning of Phase I satisfy condition~\eqref{eq:phaseI_invariant} for all $u\in B(v,t)$. At each step (except possibly the last one) of Phase I, $t$ is increased by $2$ and the point $\bm{z}$ is updated to satisfy~\eqref{eq:phaseI_invariant} for all $u\in B(v,t)$. In the end of Phase I, $t$ equals $\varepsilon(v)$, and hence~\eqref{eq:phaseI_invariant} holds for all $u\in V$. 

\smallskip

{\bf Step of Phase I:} At each step we change coordinates of point $\bm{z}$ corresponding to the nodes in $N(v, t+1)$ and $N(v, t+2)$.

If $t < b-2$, we walk from $\bm{z}$ to the point $\bm{z}'$, such that for all $u\in N(v,k)$, for all $k\leq t+1$
$$
	\bm{z}'_u=\begin{cases}
			1    & \text{  if  }  k\equiv b+1\mod 2\\
			0    & \text{  if  }  k\equiv b\mod 2\,.
		\end{cases}\,
$$
Such a point $\bm{z}'$ can be reached from $\bm{z}$ in at most two circuit steps.
From $\bm{z}'$ we walk to the point $\bm{z}''$ such that for all $u\in N(v,k)$, for all $k\leq t+2$
$$
	\bm{z}''_u=\begin{cases}
			0    & \text{  if  }  k\equiv b+1\mod 2\\
			1    & \text{  if  }  k\equiv b\mod 2\,.
		\end{cases}\,
$$
A point $\bm{z}''$ with above properties can be reached from $\bm{z}'$ in one circuit step. Thus, in this case we are able to define $\bm{z}''$ to be the new point $\bm{z}$ and increase $t$ by $2$ using at most three circuit steps.

If $t=b-2$, we walk from $\bm{z}$ to the point $\bm{z}'$, such that for all $u\in N(v,k)$, for all $k\leq t+1$
$$
	\bm{z}'_u=\begin{cases}
			1    & \text{  if  }  k\equiv b+1\mod 2\\
			0    & \text{  if  }  k\equiv b\mod 2\,.
		\end{cases}\,
$$
Such a point $\bm{z}'$ can be reached from $\bm{z}$ in at most two circuit steps. From $\bm{z}'$ we walk to the point $\bm{z}''$ such that for all $u\in N(v,k)$, for all $k\leq t+2$
$$
	\bm{z}''_u=\begin{cases}
			1/2    & \text{  if  }  k\equiv b+1\mod 2\\
			1/2    & \text{  if  }  k\equiv b\mod 2\,,
		\end{cases}\,
$$
where $k$ is such that $u\in N(v,k)$. A point $\bm{z}''$ with above properties can be reached from $\bm{z}'$ in one circuit step. From $\bm{z}''$ we walk to the point $\bm{z}'''$ such that for all $u\in N(v,k)$, for all $k\leq t+2$
$$
	\bm{z}'''_u=\begin{cases}
			0    & \text{  if  }  k\equiv b+1\mod 2\\
			1    & \text{  if  }  k < b \text{  and  } k\equiv b\mod 2\\
			1/2 & \text{  if  } k\ge b \text{  and  } k=b\mod 2\,.
		\end{cases}\,
$$
A point $\bm{z}'''$ with above properties can be reached from $\bm{z}''$ in one circuit step. Thus, in this case we are able to define $\bm{z}'''$ to be the new point $\bm{z}$ and increase $t$ by $2$ using at most four circuit steps.

If $t \ge b$, we walk from $\bm{z}$ to the point $\bm{z}'$, such that for all $u\in N(v,k)$, for all $k\leq t+1$
$$
	\bm{z}'_u=\begin{cases}
			1/2    & \text{  if  } k\equiv b+1\mod 2\\
			1/2    & \text{  if  }  k < b \text{  and  } k\equiv b\mod 2\\
			0 & \text{  if  } k\ge b \text{  and  } k\equiv b\mod 2\,.
		\end{cases}\,
$$
Such a point $\bm{z}'$ can be reached from $\bm{z}$ in one circuit step. From $\bm{z}'$ we walk to the point $\bm{z}''$ such that for all $u\in N(v,k)$, for all $k\leq t+2$
$$
	\bm{z}''_u=\begin{cases}
			0    & \text{  if  } k=b+1\mod 2\\
			1    & \text{  if  } k < b \text{  and  } k\equiv b\mod 2\\
			1/2 & \text{  if  } k\ge b \text{  and  } k\equiv b\mod 2\,.
		\end{cases}\,
$$
A point $\bm{z}''$ with above properties can be reached from $\bm{z}'$ in one circuit step. Thus, in this case we are able to define $\bm{z}''$ to be the new point $\bm{z}$ and increase $t$ by $2$ using only two circuit steps.

Note, that if at the beginning of a Phase step we have $\varepsilon(v)= t+1$, we are in the case $t\ge b$. In this case, we need only two circuits steps to update $\bm{z}$ and increase $t$ by~$1$.

\bigskip

{\bf Phase II:} We are now at the ``well structured'' point $\bm{y}'=\bm{z}$.  In this Phase, we construct a circuit walk from the current point $\bm{z}$ to the vertex $\bm{x}''$. Recall that at the end of Phase I, $\bm{z}$ satisfies~\eqref{eq:phaseI_invariant} for all $u\in V$ and $t=\varepsilon(v)$.

\smallskip

{\bf Start of Phase II:}
If for $w\in N(v,t)$ we have $\bm{z}_w=0$, then we first take two circuit steps from the current $\bm{z}$ to the point $\bm{z'}$, such that for all $u\in N(v,k)$, for all $k\leq t-1$
$$
	\bm{z}'_u=\begin{cases}
			0    & \text{  if  }  k\equiv b+1\mod 2\\
			1    & \text{  if  }  k < b \text{  and  } k\equiv b\mod 2\\
			1/2 & \text{  if  } k\ge b \text{  and  } k\equiv b\mod 2
		\end{cases}\,
$$
and for $u\in N(v,t)$
$$
	\bm{z}'_u=\begin{cases}
			1/2    & \text{  if  }  \bm{x}''_u\in \{1/2, 1\}\\
			0 & \text{  if  } \bm{x}''_u=0\,.
		\end{cases}\,
$$

Now we take two circuit steps from $\bm{z}'$ to $\bm{z}''$, such that for all $u\in N(v,k)$, for all $k\leq t-2$
$$
	\bm{z}''_u=\begin{cases}
			0    & \text{  if  }  k\equiv b+1\mod 2\\
			1    & \text{  if  }  k < b \text{  and  } k\equiv b\mod 2\\
			1/2 & \text{  if  } k\ge b \text{  and  } k\equiv b\mod 2
		\end{cases}\,
$$
and for $u\in N(v,t-1)$
we have
$$
	\bm{z}''_u=\begin{cases}
			0 & \text{  if  } uw\in E \text{  for some  } w\in N(v,t),\, \bm{x}''_w=1\\
			1/2 & \text{  otherwise  } 
		\end{cases}\,
$$
and for $u\in N(v,t)$ we have $\bm{z}''_u=\bm{x}''_u$. Thus, we define $\bm{z}''$ to be the new point $\bm{z}$ and decrease $t$ by $1$ using at most four circuit steps.

\smallskip

{\bf Invariants for $\bm{z}$ and $t$ in Phase II:}
During Phase II, we update  $\bm{z}$ and $t$ such that at each moment of time the following holds for all $u\in N(v,k)$, for all $k\leq t-1$
\begin{equation}\label{eq:phaseII_invariant1}\tag{$\star\star$}
	\bm{z}_u=\begin{cases}
			0    & \text{  if  }  k\equiv b+1\mod 2\\
			1    & \text{  if  }  k < b \text{  and  } k\equiv b\mod 2\\
			1/2 & \text{  if  } k\ge b \text{  and  } k\equiv b\mod 2
		\end{cases}\,
\end{equation}
and for $u\in N(v,t)$, we have
\begin{equation}\label{eq:phaseII_invariant2}\tag{$\star\star\star$}
	\bm{z}_u=\begin{cases}
			0    & \text{  if  }  \max\{\bm{x}''_w\,:\,w\in N(v,t+1),\, uw\in E \}=1\\
			1/2    &  \text{  if  } \max\{\bm{x}''_w\,:\,w\in N(v,t+1),\, uw\in E \}=1/2\\
			\phi    &  \text{  otherwise}\,, 
		\end{cases}\,
\end{equation}
where $\phi:=1/2$ if $b\ge t$ and  $\phi:=1$ if $b<t$. Moreover, for all $u\in N(v,k)$, $k>t$, we have $\bm{z}_u=\bm{x}''_u$. Again by construction, $\bm{z}$ and $t$ defined at the beginning of Phase II satisfy condition~\eqref{eq:phaseII_invariant1} for all $u\in B(v,t-1)$ and condition~\eqref{eq:phaseII_invariant2} for all $u\in N(v,t)$. At each step (except possibly the last one) of Phase II, $t$ is decreased by $2$ and the point $\bm{z}$ is updated to satisfy condition~\eqref{eq:phaseII_invariant1} for all $u\in B(v,t-1)$ and condition~\eqref{eq:phaseII_invariant2} for all $u\in N(v,t)$. At every moment of Phase II, we have $t=b\mod 2$.

{\bf Step of Phase II:}  

For all points on the circuit walk in a step of Phase II, we have $\bm{z}_u=\bm{x}''_u$ for every $u\in N(v,k)$, for all $k>t$.

If at the beginning of a step of Phase II  we have $t\ge b+2$, we take a circuit step from $\bm{z}$ to a point $\bm{z}'$, such that for all $u\in N(v,k)$, for all $k\leq t-1$
$$
	\bm{z}'_u=\begin{cases}
			1/2    & \text{  if  }  k\equiv b+1\mod 2\\
			1/2   & \text{  if  }  k < b \text{  and  } k\equiv b\mod 2\\
			0 & \text{  if  } k\ge b \text{  and  } k\equiv b\mod 2
		\end{cases}\,
$$
and for $u\in N(v,t)$, we have 
$$
	\bm{z}'_u=\begin{cases}
			1/2    & \text{  if  }  \bm{x}''_u\in \{1/2, 1\}\\
			0 & \text{  if  } \bm{x}''_u=0\,.
		\end{cases}\,
$$
From $\bm{z}'$ we take a circuit step to a point $\bm{z}''$, such that for all $u\in N(v,k)$, for all $k\leq t-2$
$$
	\bm{z}''_u=\begin{cases}
			0    & \text{  if  }  k\equiv b+1\mod 2\\
			1   & \text{  if  }  k < b \text{  and  } k\equiv b\mod 2\\
			1/2 & \text{  if  } k\ge b \text{  and  } k\equiv b\mod 2
		\end{cases}\,
$$
for $u\in N(v,t-1)$, we have 
$$
	\bm{z}''_u=\begin{cases}
			1/2    & \text{  if  }  \bm{x}''_u\in \{1/2, 1\}\\
			0 & \text{  if  } \bm{x}''_u=0
		\end{cases}\,
$$
and for $u\in N(v,t)$ we have $\bm{z}''_u=\bm{x}''_u$.
From $\bm{z}''$ we take a circuit step to $\bm{z}'''$ such that for all $u\in N(v,k)$, for $k\leq t-2$
$$
	\bm{z}'''_u=\begin{cases}
			1/2    & \text{  if  } k\equiv b+1\mod 2\\
			1/2   & \text{  if  } k < b \text{  and  } k\equiv b\mod 2\\
			0 & \text{  if  } k\ge b \text{  and  } k\equiv b\mod 2\,.
		\end{cases}\,
$$
Moreover, for $u\in N(v,t-1)\cup N(v,t)$ we have $\bm{z}'''_u=\bm{x}''_u$. It is not hard to see, that from $\bm{z}'''$ it takes at most one more additional circuit step to a point satisfying condition~\eqref{eq:phaseII_invariant1} for all $u\in B(v,t-3)$ and condition~\eqref{eq:phaseII_invariant2} for all $u\in N(v,t-2)$. Thus, for $t\ge b+2$ it takes at most four circuit steps to update $\bm{z}$ and decrease $t$ by $2$.

In the case when $t<b+2$, in the same way $t$ can be decreased by $2$ and the point $\bm{z}$ can be updated in at most four circuit steps. Note that for $u\in V$ we have $\bm{x}''_u=1/2$ only if $k\ge b$ or $\bm{x}''_w=1/2$ for some $w\in N(v,k+1)$, $uw\in E$, where $k$ is such that $u\in N(v,k)$. Furthermore, for the very last Phase II step we need only three circuit steps if $t=1$ and only one circuit step if $t=0$.

\bigskip

{\bf Number of Circuit Steps in the Constructed Walk:}  The total number of circuit steps needed in both Phases is at most
$4\varepsilon(v)+c$ for some constant $c$.

Indeed, to start Phase I we need at most a constant number of circuit steps. With each step of Phase I, $t$ increases by $2$ and we use at most $4$ circuit steps, until $t=\varepsilon(v)$ or $t=\varepsilon(v)-1$. In the latter case, we still need at most $4$ circuit steps to finish Phase I and increase $t$ by $1$. We also need at most $4$ circuit steps to start Phase II by updating $t$ to be equal to $\varepsilon(v)-1$.

With each step of Phase II, $t$ decreases by $2$ and we use at most $4$ circuit steps.  This is done until $t=0$ or $t=1$.  In both cases, we need an additional constant number of steps to finish Phase II. This gives the upper bound of $2\varepsilon(v)+2\varepsilon(v)+c$ for some constant $c$ on the total number of circuit steps in the constructed circuit walk from $\bm{x}'$ to $\bm{x}''$.

\end{proof}

Lemma~\ref{lem:stab ecce} immediately implies an upper bound on the diameter of the Fractional Stable Set polytope in terms of the \textit{diameter of the graph $G$}, defined as $\diam(G):=\max_{v\in V}\{\varepsilon(v)\}$.

\begin{cor}
For a connected graph $G=(V,E)$ with at least two nodes, $\cd(P_{\fstab}(G))$ is $\mathcal{O}(\diam(G))$.
\end{cor}

\bibliographystyle{plain}
\bibliography{Bibliography}
\end{document}